\documentclass[11pt,a4paper]{amsart}

\usepackage{enumitem}
\usepackage{amssymb}
\usepackage{hyperref}
\usepackage[ruled,vlined,noend]{algorithm2e}

\usepackage{multirow}
\usepackage{tikz}
\usetikzlibrary{arrows,decorations.markings,backgrounds,shapes.geometric}
\tikzstyle{point}=[circle, inner sep = 2pt, outer sep = 1pt, minimum size = 5pt, fill=black, draw=black]

\usepackage[top = 0.95in, bottom = 0.90in, left = 0.82 in, right = 0.82 in]{geometry}

\usepackage[colorinlistoftodos]{todonotes}

\renewcommand{\dim}{\operatorname{dim}}
\renewcommand{\Im}{\operatorname{Im}}
\newcommand{\rank}{\operatorname{rank}}
\newcommand{\Dgm}{\operatorname{Dgm}}
\newcommand{\Dir}{\operatorname{Dir}}
\newcommand{\dis}{\operatorname{dis}}
\newcommand{\codis}{\operatorname{codis}}
\newcommand{\CD}{\operatorname{CD}}
\newcommand{\Pers}{\operatorname{Pers}}

\newtheorem{theorem}{Theorem}[section]
\newtheorem{lemma}[theorem]{Lemma}
\newtheorem{corollary}[theorem]{Corollary}
\newtheorem{proposition}[theorem]{Proposition}

\theoremstyle{definition}

	\theoremstyle{definition}
	\newtheorem{rem&notation}[theorem]{Remark and notation}

\newtheorem{remark}[theorem]{Remark}
\newtheorem{example}[theorem]{Example}

\theoremstyle{definition}
\newtheorem{definition}[theorem]{Definition}

\numberwithin{equation}{section}

\title{A directed persistent homology theory for dissimilarity functions}

\author{David M\'endez}
\address[D.\ M\'endez]{School of Mathematical Sciences, University of Southampton, SO17 1BJ, United Kingdom \\ 
Current affiliation: Champalimaud Centre for the Unknown, Av. Bras\'ilia s/n, 1400-038 Lisboa, Portugal}
\email{david.mendez@research.fchampalimaud.org}

\author{Rub\'en J. S\'anchez-Garc\'{i}a}
\address[R.\ S\'anchez-Garc\'{i}a]{School of Mathematical Sciences, University of Southampton, SO17 1BJ, United Kingdom}
\email{R.Sanchez-Garcia@soton.ac.uk}

\makeatletter
\@namedef{subjclassname@2020}{%
  \textup{2020} Mathematics Subject Classification}
\newcommand*{\relrelbarsep}{.386ex}
\newcommand*{\relrelbar}{%
  \mathrel{%
    \mathpalette\@relrelbar\relrelbarsep
  }%
}
\newcommand*{\@relrelbar}[2]{%
  \raise#2\hbox to 0pt{$\m@th#1\relbar$\hss}%
  \lower#2\hbox{$\m@th#1\relbar$}%
}
\providecommand*{\rightrightarrowsfill@}{%
  \arrowfill@\relrelbar\relrelbar\rightrightarrows
}
\providecommand*{\leftleftarrowsfill@}{%
  \arrowfill@\leftleftarrows\relrelbar\relrelbar
}
\providecommand*{\xrightrightarrows}[2][]{%
  \ext@arrow 0359\rightrightarrowsfill@{#1}{#2}%
}
\providecommand*{\xleftleftarrows}[2][]{%
  \ext@arrow 3095\leftleftarrowsfill@{#1}{#2}%
}
\makeatother

\begin{document}

\keywords{Persistent homology, dissimilarity functions, directed simplicial complexes, directed cycles}

\begin{abstract}
    We develop a theory of persistent homology for directed simplicial complexes which detects persistent directed cycles in odd dimensions.  We relate directed persistent homology to classical persistent homology, prove some stability results, and discuss the computational challenges of our approach. Our directed persistent homology theory is motivated by homology with semiring coefficients: by explicitly removing additive inverses, we are able to detect directed cycles algebraically.
\end{abstract}

\maketitle

\section{Introduction}
%intro, abstract, section structure, intro to sections
Persistent homology is one of the most successful tools in Topological Data Analysis \cite{Car09}, with recent applications in numerous scientific domains such as biology, medicine, neuroscience, robotics, and many others \cite{OttPorTil+17}. In its most common implementation, persistent homology is used to infer topological properties of the metric space underlying a finite point cloud using two steps \cite{EdeHar10}:

\begin{enumerate}
    \item Build a filtration of simplicial complexes from distances, or similarities, between data points.
    
    \item Compute the singular homology of each of the simplicial complexes in the filtration, along with the linear maps induced in homology by their inclusions. The resulting \emph{persistence module} can normally be represented using a \emph{persistence diagram} or a \emph{persistence barcode}.
\end{enumerate}
A fundamental limitation of persistent homology, and in fact homology, is its inability to incorporate directionality, which can be important in some real-world applications (see examples below). For instance, homology cannot in principle distinguish between directed and undirected cycles (Figure \ref{figure:directedCycles}). Although the definition of homology, namely the differential or boundary operator, requires a choice of orientation for the simplices, the resulting homology is independent of this choice. Previous attempts have had some partial success at this issue (see `Related work' below), however, as far as we know, there is no homology theory able to exactly detect directed cycles up to boundary equivalence. 

The main difficulty occurs at the algebraic level: opposite orientations on a simplex correspond to additive inverse elements in the coefficient ring or field. Our key insight is to retain the submodule of homology generated by those classes that only have non-negative coefficients on elementary chains. In this way, we can define a directed homology module (Def.~\ref{definition:directed-homology}) of directed simplicial complexes (Def.~\ref{definition:directedSimplicialComplex}) that detects homology classes of directed 1-cycles, as desired  (Fig.~\ref{figure:directedCycles}, Prop.~\ref{proposition:homologyPolygons}).

This insight arises from the use of homology with semiring coefficients, where additive inverses may not exist. Indeed, the main results in this paper can be understood and reformulated in full generality using homology theory with semiring coefficients, see Appendix \ref{appendix:algebraicBackground}.

Our directed homology theory can be extended to the persistent setting. More concretely, we define two persistence modules associated to filtrations of directed simplicial complexes: an undirected one using the whole homology module (Def.~\ref{definition:undirPersistenceModFilt}) and a directed one using the submodule of directed homology (Def.~\ref{definition:dirPersistenceModFilt}). We show that the bars in the directed barcodes can be matched one-to-one with bars in the undirected barcodes, with the directed bars possibly having a later birth (Prop.~\ref{proposition:subbarcode}, Figs.~\ref{figure:severalAtOnce}, \ref{figure:filtrationOne}, \ref{figure:notEveryDirected}).
  
Our motivation to develop directed persistent homology is the applications to real-world data sets where asymmetry or directionality, in particular the presence of directed cycles, plays a fundamental role. Examples include biological neural networks (directed synaptic connections), time series data (directed temporal connections) or biological molecular networks such as protein-protein interaction networks (inhibitory and excitatory connections). Mathematically, we can model these situations using a finite set $V$ of data points, and an arbitrary function $d_V \colon V \times V \to \mathbb{R}$, which we call \emph{dissimilarity function} (Section \ref{section:introAsymmetricNetrowks}) and which, crucially, may not be symmetric. Our main example of a directed simplicial complex is the directed Rips complex (Definition \ref{definition:filtrationDisNet}) of $(V,d_V)$, which becomes the input of our suggested directed persistent homology pipeline. 

	\begin{figure}
		\begin{center}
			\begin{tikzpicture}
				\tikzset{node distance = 1.5cm, auto}
				\node[style=point,label={[label distance=0cm]270:$v_1$}](v1) {};
				\node[style=point, right of = v1, label={[label distance = 0cm]270:$v_2$}](v2) {};
				\node[style=point, above of = v1, xshift = 0.75cm, yshift = -0.2 cm, label={[label distance = 0cm]90:$v_3$}](v3) {};
				\draw[->, thick, >=stealth] (v1) to node{} (v2);
				\draw[->, thick, >=stealth] (v1) to node{} (v3);
				\draw[->, thick, >=stealth] (v2) to node{} (v3);
%%%%%%%%%%%%%%%%%%%%%%%%%%%%%%%%%%%%%%%%%%%%%%%%
				\node[style=point, right of = v2, xshift = 1cm, label={[label distance = 0cm]270:$w_1$}](w1) {};
				\node[style=point, right of = w1, label={[label distance = 0cm]270:$w_2$}](w2) {};
				\node[style=point, above of = w1, xshift = 0.75cm, yshift = -0.2 cm, label={[label distance = 0cm]90:$w_3$}](w3) {};
				\draw[->, thick, >=stealth] (w1) to node{} (w2);
				\draw[->, thick, >=stealth] (w3) to node{} (w1);
				\draw[->, thick, >=stealth] (w2) to node{} (w3);
			\end{tikzpicture}
			\caption{Two examples of directed simplicial complexes $X$ (left) and $Y$ (right). Our directed submodules of homology over the rings $R\in\{\mathbb{Z},\mathbb{Q},\mathbb{R}\}$ detect directed 1-cycles: $H_1^{\Dir}(X,R)=\{0\}$, while $H_1^{\Dir}(Y,R)=R$ generated by the directed 1-cycle $[w_1,w_2] + [w_2,w_3] + [w_3,w_1]$.
			}
			\label{figure:directedCycles}
	\end{center}
\end{figure}
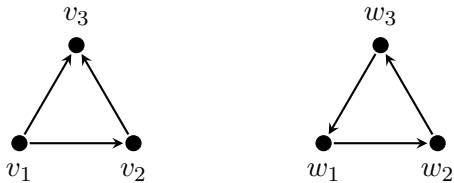

In the context of dissimilarity functions, we prove a stability result with respect to the bottleneck distance between (directed) persistent diagrams and the correspondence distortion distance (a generalisation of the Gromov-Hausdorff distance) between dissimilarity functions over finite sets (Thm.~\ref{theorem:stability}). We finish this article by discussing the computational challenges of our approach. 

\subsubsection*{Related work} A first approach to incorporating directionality would be to encode the asymmetry of a data set in the simplicial complex built from it. In \cite{Tur19}, the author uses \emph{ordered tuple complexes} or \emph{OT-complexes}, which are generalisations of simplicial complexes where simplices are ordered tuples of vertices. Similar ideas have been successfully used in the field of neuroscience to show the importance of \emph{directed cliques} of neurons \cite{ReiNolMar+17}, showing that directionality in neuron connectivity plays a crucial role in the structure and function of the brain. 
In \cite{ChoMem18}, the authors use so-called \emph{Dowker filtrations} to develop persistent homology for asymmetric networks, and note \cite[Remark 35]{ChoMem18} that a non-trivial 1-dimensional persistence diagram associated to Dowker filtrations suggests the presence of directed cycles. However, they also note that such persistence diagram may be non-trivial even if directed cycles are not present.
Further progress can be achieved by combining the ideas above with homology theories that are themselves sensitive to asymmetry. In \cite{ChoMem18b}, the authors develop persistent homology for directed networks using the theory of path homology of graphs \cite{GriLinMurYau12}. This persistent homology theory shows stability with respect to the distance between asymmetric networks \cite{CarMemRibSeg14}, and is indeed able to tell apart digraphs with isomorphic underlying graphs but different orientations on the edges. Nonetheless, cycles in path homology do not correspond to directed cycles, see \cite[Example 10]{ChoMem18b}.
Another significant contribution can be found in \cite{Tur19}, where four new approaches to persistent homology are developed for asymmetric data, all of which are shown to be stable. Each approach is sensitive to asymmetry in a different way. For example, one of them uses a generalisation of poset homology to preorders to detect strongly connected components in digraphs, a feature our implementation does not have (see Proposition \ref{proposition:H0andConnectivity}).
More interestingly for our purposes, in \cite[Section 5]{Tur19} the author introduces a persistent homology approach that builds a \emph{directed Rips filtration} of ordered tuple complexes associated to \emph{dissimilarity functions} $d_V\colon V\times V\to\mathbb{R}$. This is indeed sensitive to asymmetry and yields a persistent homology pipeline that is stable with respect to the correspondence distortion distance, a generalisation of the Gromov-Hausdorff distance to  dissimilarity functions (see \cite[Section 2]{Tur19} and Section \ref{section:introAsymmetricNetrowks}). However, undirected cycles are also detected by this homology theory. %Finally, note that semirings have been used to gather insights on other directed problems, such as algebraic shortest-path problems \cite{Moh02}, and in automata theory \cite{KuiSal85}.
In contrast, in this article we produce a theory of persistent homology associated to ordered tuple complexes that only detects directed cycles and which is also stable with respect to the correspondence distortion distance.

\subsubsection*{Overview of results}

We introduce modules of undirected and directed homology associated to directed simplicial complexes (Definition \ref{definition:directedHomologyComplexes}) and show that they satisfy some desirable properties. Furthermore, we prove that for certain coefficient rings, closed paths will only have non-trivial directed homology in dimension $1$ when their edges form a directed cycle, as we set up to do (Proposition \ref{proposition:homologyPolygons}). Our modules of directed homology can be defined in dimensions greater than one, although they are zero in all even dimensions (Proposition \ref{proposition:evenDimTrivialHomology}) except in dimension 0 where they detect weakly connected components (Proposition \ref{proposition:H0andConnectivity}). 

Next, we extend our homology theory to the persistence setting. In particular, we introduce, for each dimension, two persistence modules associated to the same filtration: an \emph{undirected} persistence module (Definition \ref{definition:undirPersistenceModFilt}), which is analogous to the one introduced in \cite[Section 5]{Tur19}, and the submodule generated by directed classes, which we call the \emph{directed} persistence module (Definition \ref{definition:dirPersistenceModFilt}). The directed persistence module gives rise to bars associated to homology classes that can be represented by directed cycles, as illustrated in Fig.~\ref{figure:severalAtOnce} and Examples \ref{example:filtrationEquivalentCycles} to \ref{example:remainsUnmatched}.
We also establish a relation between the undirected and directed persistence barcodes of the same filtration. Indeed, in Proposition \ref{proposition:subbarcode} we show that every bar in the directed barcode corresponds to a unique bar in the undirected barcode that dies at the same time. The bar in the directed barcode may, however, be born later, and some bars in the undirected barcode may be left unmatched (see Figs.~ \ref{figure:severalAtOnce}, \ref{figure:filtrationOne} and \ref{figure:notEveryDirected}). 

\begin{figure}
    \centering
    \begin{tikzpicture}[mystyle/.style={draw,shape=circle,fill=blue}]
        \def\ngon{5}
        \node[regular polygon,regular polygon sides=\ngon,minimum size=2.2cm,shape border rotate=216] (p) {};
        \node[regular polygon,regular polygon sides=\ngon,minimum size=2.9cm,shape border rotate=216] (q) {};
        \foreach\x in {1,...,\ngon}{\node[point] (p\x) at (p.corner \x){};}
        \foreach\x in {1,...,\ngon}{\node (q\x) at (q.corner \x){$v_{\x}$};}
        \draw[->,thick,>=stealth] (p1) to node {} (p2);
        \draw[->,thick,>=stealth] (p2) to node {} (p4);
        \draw[->,thick,>=stealth] (p2) to node {} (p3);
        \node[below of = p3, yshift=-1.7cm] (text) {$X_0$};
    \end{tikzpicture}
    \qquad
    \begin{tikzpicture}[mystyle/.style={draw,shape=circle,fill=blue}]
        \def\ngon{5}
        \node[regular polygon,regular polygon sides=\ngon,minimum size=2.2cm,shape border rotate=216] (p) {};
        \node[regular polygon,regular polygon sides=\ngon,minimum size=2.9cm,shape border rotate=216] (q) {};
        \foreach\x in {1,...,\ngon}{\node[point] (p\x) at (p.corner \x){};}
        \foreach\x in {1,...,\ngon}{\node (q\x) at (q.corner \x){$v_{\x}$};}
        \draw[->,thick,>=stealth] (p1) to node {} (p2);
        \draw[->,thick,>=stealth] (p2) to node {} (p4);
        \draw[->,thick,>=stealth] (p2) to node {} (p3);
        \draw[->,thick,>=stealth] (p3) to node {} (p4);
        \draw[->,thick,>=stealth] (p3) to node {} (p5);
        \node[below of = p3, yshift=-1.7cm] (text) {$X_1$};
    \end{tikzpicture}
    \qquad
    \begin{tikzpicture}[mystyle/.style={draw,shape=circle,fill=blue}]
        \def\ngon{5}
        \node[regular polygon,regular polygon sides=\ngon,minimum size=2.2cm,shape border rotate=216] (p) {};
        \node[regular polygon,regular polygon sides=\ngon,minimum size=2.9cm,shape border rotate=216] (q) {};
        \foreach\x in {1,...,\ngon}{\node[point] (p\x) at (p.corner \x){};}
        \foreach\x in {1,...,\ngon}{\node (q\x) at (q.corner \x){$v_{\x}$};}
        \draw[->,thick,>=stealth] (p1) to node {} (p2);
        \draw[->,thick,>=stealth] (p2) to node {} (p4);
        \draw[->,thick,>=stealth] (p2) to node {} (p3);
        \draw[->,thick,>=stealth] (p3) to node {} (p4);
        \draw[->,thick,>=stealth] (p3) to node {} (p5);
        \draw[->,thick,>=stealth] (p4) to node {} (p5);
        \draw[->,thick,>=stealth] (p1) to node {} (p3);
        \node[below of = p3, yshift=-1.7cm] (text) {$X_2$};
    \end{tikzpicture}
    \qquad
    \begin{tikzpicture}[mystyle/.style={draw,shape=circle,fill=blue}]
        \def\ngon{5}
        \node[regular polygon,regular polygon sides=\ngon,minimum size=2.2cm,shape border rotate=216] (p) {};
        \node[regular polygon,regular polygon sides=\ngon,minimum size=2.9cm,shape border rotate=216] (q) {};
        \foreach\x in {1,...,\ngon}{\node[point] (p\x) at (p.corner \x){};}
        \foreach\x in {1,...,\ngon}{\node (q\x) at (q.corner \x){$v_{\x}$};}
        \draw[->,thick,>=stealth] (p1) to node {} (p2);
        \draw[->,thick,>=stealth] (p2) to node {} (p4);
        \draw[->,thick,>=stealth] (p2) to node {} (p3);
        \draw[->,thick,>=stealth] (p3) to node {} (p4);
        \draw[->,thick,>=stealth] (p3) to node {} (p5);
        \draw[->,thick,>=stealth] (p4) to node {} (p5);
        \draw[->,thick,>=stealth] (p1) to node {} (p3);
        \draw[->,thick,>=stealth] (p5) to node {} (p1);
        \node[below of = p3, yshift=-1.7cm] (text) {$X_3$};
    \end{tikzpicture}
    
    \bigskip
    
    \begin{tikzpicture}
        \node at (0,-0.3){0};
        \node at (1,-0.3){1};
        \node at (2,-0.3){2};
        \node at (3,-0.3){3};
        \draw[->, >=stealth] (0,0)--(4,0);
        \draw[-] (0,-0.07)--(0,0.07);
        \draw[-] (1,-0.07)--(1,0.07);
        \draw[-] (2,-0.07)--(2,0.07);
        \draw[-] (3,-0.07)--(3,0.07);
        \draw[->,>=stealth,very thick,color=blue](1,0.3)--(4,0.3);
        \draw[->,>=stealth,very thick,color=red](2,0.6)--(4,0.6);
        \draw[->,>=stealth,very thick,color=green](2,0.9)--(4,0.9);
        \draw[->,>=stealth,very thick,color=orange](3,1.2)--(4,1.2);
    \end{tikzpicture}
    \qquad\qquad
    \begin{tikzpicture}
        \node at (0,-0.3){0};
        \node at (1,-0.3){1};
        \node at (2,-0.3){2};
        \node at (3,-0.3){3};
        \draw[->, >=stealth] (0,0)--(4,0);
        \draw[-] (0,-0.07)--(0,0.07);
        \draw[-] (1,-0.07)--(1,0.07);
        \draw[-] (2,-0.07)--(2,0.07);
        \draw[-] (3,-0.07)--(3,0.07);
        \draw[->,>=stealth,very thick,color=blue](3,0.3)--(4,0.3);
        \draw[->,>=stealth,very thick,color=red](3,0.6)--(4,0.6);
        \draw[->,>=stealth,very thick,color=green](3,0.9)--(4,0.9);
        \draw[->,>=stealth,very thick,color=orange](3,1.2)--(4,1.2);
    \end{tikzpicture}
        
    \caption{A filtration of a directed simplicial complex (top, left to right) and corresponding undirected (bottom, left) and directed (bottom, right) persistence barcodes. Every bar in the directed barcode corresponds to a unique bar in the undirected barcode (shown here by matching colours).}
    \label{figure:severalAtOnce}
\end{figure}
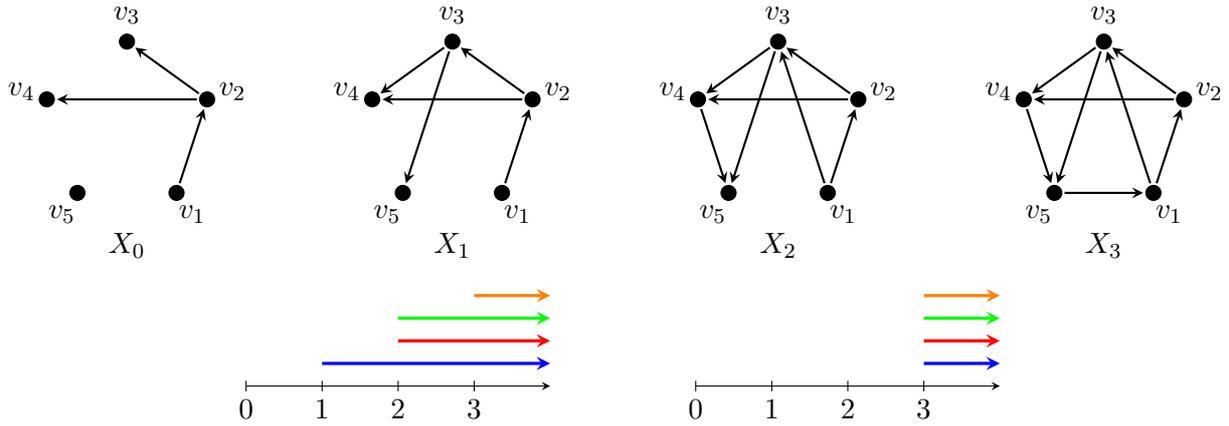

%This has an easy interpretation. Namely, if advancing in the filtration causes the birth of a directed cycle, it can be equivalent to a combination of preexisting undirected classes. However, a class admitting a directed representative must admit one such representative as long as it is not made equivalent to the trivial class. This behaviour can be observed in Figure \ref{figure:severalAtOnce}. Note, however, that not every class will necessarily become directed, or equivalently, there may be bars in the undirected barcode which remain unmatched throughout the filtration, see Example \ref{example:remainsUnmatched}.

Having all the necessary ingredients, we can provide a complete workflow for directed persistent homology: given a dissimilarity function $d_V$ on a finite set $V$, we construct its directed Rips filtration (Definition \ref{definition:filtrationDisNet}), and take the corresponding undirected and directed $n$-dimensional persistence modules, respectively denoted $\mathcal{H}_n(V, d_V)$ and $\mathcal{H}_{n}^{\Dir}(V,d_V)$ (Definition \ref{definition:persistenceNetwork}). Both of these persistence modules have persistence diagrams and barcodes (Definition \ref{definition:directedPersBarcodes}) and the associated persistence modules are stable with respect to the correspondence distortion distance (Section \ref{section:introAsymmetricNetrowks} and Theorem \ref{theorem:stability}).

We finish the article with a discussion of the computational challenges of calculating the directed persistent homology of a dissimilarity function. The Standard Algorithm can easily be adapted to compute the barcodes of the undirected persistent homology of the dissimilarity function, but the computation of the directed persistence barcodes is more challenging. A direct approach results in an extreme ray enumeration problem (Section \ref{section:algorithms}), a well-known problem in computational geometry. Developing an algorithm that computes the directed homology classes, rather than all directed cycles, remains an open problem.

We include an Appendix on homology with semiring coefficients and on how it gives rise to the ideas introduced in this article. This provides additional motivation and context for our results, and presents them in full generality. 
%and a potential source for more efficient algorithms to compute directed persistent homology.

\subsubsection*{Outline of the article} In Section \ref{section:introPersistentNetworks}, we introduce the necessary background in persistence modules (\ref{section:persistentHomologyIntro}) and dissimilarity functions (\ref{section:introAsymmetricNetrowks}). In Section \ref{section:homologySimplicialComplexes} we introduce undirected and directed homology modules of directed simplicial complexes. We then use these homology modules to introduce directed and undirected persistent homology in Section \ref{section:directedPersistentHomology}, where we also show our stability results with respect to dissimilarity functions.
In Section \ref{section:algorithms}, we discuss the computational implementation of directed persistent homology and lay out some future research directions. Finally, Appendix \ref{appendix:algebraicBackground} introduces homology with semiring coefficients and shows how it relates to directed homology.

\section{Persistence modules and dissimilarity functions}\label{section:introPersistentNetworks}

Our goal is to extend persistent homology to directed simplicial complexes such as those constructed from a dissimilarity function. 
In this section, we introduce the necessary background regarding persistent homology (Section \ref{section:persistentHomologyIntro}) and dissimilarity functions (Section \ref{section:introAsymmetricNetrowks}). 

\subsection{Persistence modules, diagrams and barcodes}\label{section:persistentHomologyIntro}
In this section, we introduce persistence modules and their associated persistence diagrams and barcodes. We follow the exposition in \cite{ChoMem18b}, as it is simple and focused on persistence modules arising in the context we are interested in: that of persistent homology of finite simplicial complexes. Many of the results below hold in more generality \cite{ChaSilGliOud16} but we chose to keep the exposition simple.

Let $R$ be a ring with unity and let $T\subseteq \mathbb{R}$ be a subset.

\begin{definition}%[{\cite[Definition 3.3]{Car14}}]
	A \emph{persistence $R$-module} over $T$, written $\mathcal{V} = \big(\{V^\delta\},\{\nu_\delta^{\delta'}\}\big)_{\delta\le\delta'\in T}$, is a family of $R$-modules $\{V^\delta\}_{\delta\in T}$ and homomorphisms $\nu_\delta^{\delta'}\colon V^\delta\to V^{\delta'}$, whenever $\delta\le\delta'\in T$, such that
	\begin{enumerate}
	    \item for every $\delta\in T$, $\nu_\delta^\delta$ is the identity map, and
	    \item for every $\delta\le\delta'\le\delta''\in T$, $\nu_{\delta'}^{\delta''}\circ\nu_\delta^{\delta'} = \nu_\delta^{\delta''}$.
	\end{enumerate}

	Let  $\mathcal{V} = \big(\{V^\delta\},\{\nu_\delta^{\delta'}\}\big)_{\delta\le\delta'\in T}$ and $\mathcal{W} = \big(\{W^\delta\},\{\mu_\delta^{\delta'}\}\big)_{\delta\le\delta'\in T}$ be two persistence $R$-modules over $T$. A \emph{morphism of persistence $R$-modules} $f\colon\mathcal{V}\to\mathcal{W}$ is a family of morphisms of $R$-modules $\{f^\delta\colon V^\delta\to W^\delta\}_{\delta\in T}$ such that for every $\delta\le\delta'\in T$ we have a commutative diagram
	\begin{center}
		\begin{tikzpicture}
			\tikzset{node distance = 2cm, auto}
				\node (1) {$V^\delta$};
				\node[right of = 1] (2) {$V^{\delta'}$};
				\node[below of = 1,yshift=0.3cm] (3) {$W^{\delta}$};
				\node[right of = 3] (4) {$W^{\delta'}$.};
				\draw[->] (1) to node {$\nu_\delta^{\delta'}$} (2);
				\draw[->] (3) to node {$\mu_\delta^{\delta'}$} (4);
				\draw[->] (1) to node {$f^\delta$} (3);
				\draw[->] (2) to node {$f^{\delta'}$} (4);
		\end{tikzpicture}
	\end{center}
\end{definition}

Let us assume that $R$ is a field, thus $\mathcal{V} = \big(\{V^\delta\},\{\nu_\delta^{\delta'}\}\big)_{\delta\le\delta'\in T}$ is a persistence vector space, and that $V^\delta$ is finite-dimensional, for every $\delta\in \mathbb{R}$. Furthermore, we suppose that there exists a finite subset $\{\delta_0,\delta_1,\dots,\delta_n\}\subseteq T$ for which
\begin{enumerate}
    \item if $\delta\in T$, $\delta\le\delta_0$, then $V^\delta = 0$,
    \item if $\delta\in T\cap [\delta_{i-1},\delta_i)$ for some $1\le i \le n$, the map $\nu_{\delta_{i-1}}^\delta$ is an isomorphism, while $\nu_\delta^{\delta_i}$ is not, and
    \item if $\delta,\delta'\in T$, $\delta_n \le \delta < \delta'$, then $v_\delta^{\delta'}\colon V^{\delta}\to V^{\delta'}$ is an isomorphism.
\end{enumerate}
\begin{remark}
When $R$ is a field, either of these restrictions alone (namely, each $V^\delta$ being finite-dimensional or the existence of a finite subset $\{\delta_0,\delta_1,\dots,\delta_n\}$ as above) is enough to associate a \emph{persistence barcode} and a \emph{persistence diagram} to $\mathcal{V}$ \cite[Theorem 2.8]{ChaSilGliOud16}. We nevertheless assume both restrictions since they hold for persistence diagrams associated to finite simplicial complexes, which are precisely the sort of objects that arise from data, and such assumption makes the exposition simpler.
\end{remark}
Let us now introduce persistence barcodes and diagrams. First, for simplicity, we consider $\mathcal{V}$ indexed by the natural numbers: 
\[\mathcal{V} = \big(\{V^{\delta_i}\},\{\nu_{\delta_i}^{\delta_{i+1}}\}\big)_{i\in\mathbb{N}},\]
where $V^{\delta_k} = V^{\delta_n}$ for all $k \ge n$ and $\nu_{\delta_k}^{\delta_l}$ is the identity whenever $k, l \ge n$. (This clearly contains all of the information in $\mathcal{V}$.) 

By \cite[Basis Lemma]{EdeJabMro15}, we can find bases $B_i$ of the vector spaces $V^{\delta_i}$, $i \in \mathbb{N}$, such that
\begin{enumerate}
	\item $\nu_{\delta_i}^{\delta_{i+1}}(B_i)\subseteq B_{i+1}\cup\{0\}$,
	\item $\rank(\nu_{\delta_i}^{\delta_{i+1}}) = |v_{\delta_i}^{\delta_{i+1}}(B_i)\cap B_{i+1}|$, and
	\item each $w\in \Im\big(v_{\delta_i}^{\delta_{i+1}}\big)\cap B_{i+1}$ is the image of exactly one element $v\in B_i$.
\end{enumerate} 
Such bases are called \emph{compatible bases}. Elements in $B_i$ that are mapped to an element in $B_{i+1}$ correspond to linearly independent elements of $V^{\delta_i}$ that `survive' until the next step in the persistence vector space. Similarly, elements in a basis $B_i$ which are not in $B_{i-1}$ are considered to be `born' at index $i$. Formally, we define
\[L:=\{(b,i)\mid b\in B_i, b\not\in\Im(\nu_{\delta_{i-1}}^{\delta_i}),i > 0\}\cup\{(b,0)\mid b\in B_0\}.\]
Given $(b,i) \in L$, we call $i$ the \emph{birth index} of the basis element $b$. This element `survives' on subsequent bases until it is eventually mapped to zero. When this happens, the number of steps taken until the class dies is its \emph{death index} of $b$. Formally, the death index of $(b,i)$ is 
\[\ell(b,i) :=\max\{j\in\mathbb{N}\mid (\nu_{\delta_{j-1}}^{\delta_j}\circ\dots\circ\nu_{\delta_{i+1}}^{\delta_{i+2}}\circ\nu_{\delta_i}^{\delta_{i+1}})(b)\in B_j\}.\]
We allow $\ell(b,i) = \infty$, if $b$ is in every basis $B_j$ for all $j\ge i$, so that the death index takes values in $\overline{\mathbb{N}}=\mathbb{N} \cup \{+\infty\}$. Using this information, we can introduce the \emph{persistence barcode} of $\mathcal{V}$. We call a pair $(X,m)$ a \emph{multiset} if $X$ is a set and $m \colon X \to \overline{\mathbb{N}}=\mathbb{N}\cup \{+\infty\}$ is a function. We call $m(x)$ the \emph{multiplicity} of $x \in X$.
 
\begin{definition}\label{definition:persistence barcode}
	Consider a persistence vector space $\mathcal{V} = \big(\{V^{\delta_i}\},\{\nu_{\delta_i}^{\delta_{i+1}}\}\big)_{i\in\mathbb{N}}$ 
	as above. The \emph{persistence barcode of $\mathcal{V}$} is defined as the multiset of intervals
	\[\Pers(\mathcal{V}) := \big\{[\delta_i,\delta_{j+1})\mid\text{$\exists(b,i)\in L, j\in\mathbb{N}$ s.t.
	$\ell(b,i)=j$}\big\}\cup\big\{[\delta_i,+\infty)\mid\text{$\exists(b,i)\in L$ s.t. $\ell(b,i)=+\infty$}\big\},\]
	with the multiplicity of $[\delta_i,\delta_{j+1})$ (respectively $[\delta_i,+\infty)$) being the number of elements $(b,i)\in L$ such that $\ell(b,i) = j$ (respectively $\ell(b,i)=+\infty$). 
\end{definition}

Thus, persistence barcodes encode the birth and death of elements in a family of compatible bases. Crucially, and even though compatible bases are not unique, the number of birth and death events at each step is determined by the rank of the linear maps $\nu_{\delta_i}^{\delta_{i+1}}$ and their compositions, and it is thus independent of the choice of compatible bases. Furthermore, the persistence barcode of a persistence vector space completely determines both the dimension of the vector spaces $V^{\delta_i}$ and the ranks of the linear maps between them, hence it determines the persistence vector space up to isomorphism \cite[Theorem 2.8]{ChaSilGliOud16}.

Each interval in $\Pers(\mathcal{V})$ is called a \emph{persistence interval}. Persistence barcodes can be represented by stacking horizontal lines, each of which represents a persistence interval or bar. The endpoints of the line in the horizontal axis correspond to the endpoints of the interval it represents, whereas the vertical axis has no significance other than being able to represent every persistence interval at once. Persistence bars can be stacked in any order, although they are usually ordered by their birth. This representation is what gives persistence barcodes their name.

An alternative characterisation of a persistence vector space is its \emph{persistence diagram}. Let us write $\overline{\mathbb{R}}=\mathbb{R}\cup\{-\infty,+\infty\}$ for the extended real line.

\begin{definition}\label{definition:persistenceDiagram}
    The \emph{persistence diagram} of the persistence vector space $\mathcal{V}$ is the multiset
	\[\Dgm(\mathcal{V}) := \big\{(\delta_i, \delta_{j+1})\in\overline{\mathbb{R}}^2\mid[\delta_i,\delta_{j+1})\in\Pers(\mathcal{V})\big\}\cup\big\{(\delta_i,+\infty)\in\overline{\mathbb{R}}^2\mid [\delta_i,+\infty)\in\Pers(\mathcal{V})\big\}.\]
	The multiplicity of a point in $\Dgm(\mathcal{V})$ is the multiplicity of the corresponding interval in $\Pers(\mathcal{V})$.
\end{definition} 

%Note that points in a persistence diagram are always on or above the diagonal of $\overline{\mathbb{R}}^2$.

A crucial property for applications of persistent homology to real data is that a small perturbation of the input (a data set, encoded as filtration of simplicial complexes) results in a small perturbation of the output (its persistence module). Algebraically, this amounts to proving that a small perturbation of the persistence diagram results in a small perturbation of the associated persistence module. In order to state this stability result, we therefore need to introduce distances between persistence diagrams, and persistence modules, respectively. 

To measure how far apart two persistence diagrams are, we can use the bottleneck distance between multisets of $\overline{\mathbb{R}}^2$. Let $\Delta^\infty$ denote the multiset of $\overline{\mathbb{R}}^2$ consisting of every point in the diagonal counted with infinite multiplicity. A bijection of multisets $\varphi\colon (X,m_X)\to (Y,m_Y)$ is a bijection of sets $\varphi\colon \cup_{x\in X} \sqcup_{i=1}^{m_X(x)} x\to \cup_{y\in Y} \sqcup_{i=1}^{m_Y(y)} y$, that is, a bijection between the sets obtained when counting each element in both $X$ and $Y$ with its multiplicity.

\begin{definition}\label{definition:bottleneckDistance}
	Let $A$ and $B$ be two multisets in $\overline{\mathbb{R}}^2$. The \emph{bottleneck distance} between $A$ and $B$ is defined as
	\[d_B(A,B):=\inf_\varphi\left\{\sup_{a\in A}\|a-\varphi(a)\|_\infty \right\},\]
where the infimum is taken over all bijections of multisets $\varphi\colon A\cup\Delta^\infty\to B\cup\Delta^\infty$ and where $\|\cdot\|_\infty$ denotes the $\ell^\infty$-norm in $\overline{\mathbb{R}}^2$. 

Note that the bottleneck distance between two persistence diagrams defined from persistence barcodes can only be finite if the diagrams have the same amount of bars of form $[a,\infty)$ for $a\in\mathbb{R}$.
\end{definition}

We also need a distance between persistence vector spaces. To that purpose, we use the \emph{interleaving distance}, introduced in \cite{ChaChoGli+09}.  

\begin{definition}\label{definition:interleaving}
	Let $\mathcal{V} = \big(\{V^\delta\},\{\nu_\delta^{\delta'}\}\big)_{\delta\le\delta'\in \mathbb{R}}$ and $\mathcal{W} = \big(\{W^\delta\},\{\mu_\delta^{\delta'}\}\big)_{\delta\le\delta'\in \mathbb{R}}$ be two persistence vector spaces, and $\varepsilon \ge 0$. We say that $\mathcal{V}$ and $\mathcal{W}$ are \emph{$\varepsilon$-interleaved} if there exist two families of linear maps
	\[\{\varphi_\delta\colon V^\delta\to W^{\delta+\varepsilon}\}_{\delta\in\mathbb{R}},\quad \{\psi_\delta\colon W^\delta\to V^{\delta+\varepsilon}\}_{\delta\in\mathbb{R}}\]
	such that the following diagrams are commutative for all $\delta'\ge\delta\in\mathbb{R}$:
	\begin{center}\begin{tikzpicture}
		\tikzset{node distance = 2.5cm, auto}
		\node(1) {$V^\delta$};
		\node[right of = 1,xshift=0.3cm] (2) {$V^{\delta'}$};
		\node[below of = 1] (3) {$W^{\delta + \varepsilon}$};
		\node[below of = 2] (4) {$W^{\delta'+\varepsilon}$,};
		\node[right of = 2, xshift = 1cm] (5) {$W^{\delta}$};
		\node[right of = 5,xshift = 0.3cm] (6) {$W^{\delta'}$};
		\node[below of = 5] (7) {$V^{\delta+\varepsilon}$};
		\node[below of = 6] (8) {$V^{\delta'+\varepsilon}$,};
		\draw[->] (1) to node{$\nu_\delta^{\delta'}$} (2);
		\draw[->] (3) to node{$\mu_{\delta+\varepsilon}^{\delta'+\varepsilon}$} (4);
		\draw[->] (1) to node {$\varphi_\delta$} (3);
		\draw[->] (2) to node {$\varphi_{\delta'}$} (4);
		\draw[->] (5) to node{$\mu_\delta^{\delta'}$} (6);
		\draw[->] (7) to node{$\nu_{\delta+\varepsilon}^{\delta'+\varepsilon}$} (8);
		\draw[->] (5) to node {$\psi_\delta$} (7);
		\draw[->] (6) to node {$\psi_{\delta'}$} (8);
	\end{tikzpicture}\end{center}
	\begin{center}\begin{tikzpicture}
		\tikzset{node distance = 2.5cm, auto}
		\node(1) {$V^\delta$};
		\node[right of = 1, xshift=2cm] (3) {$V^{\delta+2\varepsilon}$};
		\node[below of = 1,xshift=2.25cm] (4) {$W^{\delta + \varepsilon}$,};
		\node[right of = 3, xshift = 1cm](5) {$W^\delta$};
		\node[right of = 5,xshift=2cm] (7) {$V^{\delta+2\varepsilon}$};
		\node[below of = 5, xshift = 2.25cm] (8) {$W^{\delta + \varepsilon}$.};
		\draw[->] (1) to node{$\nu_\delta^{\delta+2\varepsilon}$} (3);
		\draw[->] (1) to node{$\varphi_\delta$} (4);
		\draw[->] (4) to node {$\psi_{\delta+\varepsilon}$} (3);
		\draw[->] (5) to node{$\mu_\delta^{\delta+2\varepsilon}$} (7);
		\draw[->] (5) to node{$\psi_\delta$} (8);
		\draw[->] (8) to node {$\varphi_{\delta+\varepsilon}$} (7);
	\end{tikzpicture}\end{center}

	The \emph{interleaving distance} between $\mathcal{V}$ and $\mathcal{W}$ is then defined as
	\[d_I(\mathcal{V},\mathcal{W}) = \inf\{\varepsilon\ge 0 \mid\text{$\mathcal{V}$ and $\mathcal{W}$ are $\varepsilon$-interleaved}\}.\]
\end{definition}

%This notion of distance is inspired by the functional setting. Namely, if $f$ and $g$ are real valued functions for which $\|f-g\|_\infty\le\varepsilon$, then $f^{-1}\big((-\infty,\delta]\big)\subseteq g^{-1}\big((-\infty,\delta+\varepsilon]\big)\subseteq f^{-1}\big((-\infty,\delta+2\varepsilon]\big)$ for every $\delta\in\mathbb{R}$. These inclusions, together with the canonical inclusions of the sublevel sets of both $f$ and $g$, yield commutative diagrams in homology analogous to those in Definition \ref{definition:interleaving}.

The authors in \cite{ChaChoGli+09} show that the interleaving distance is a pseudometric (a zero distance between distinct points may occur) in the class of persistence vector spaces. Moreover, they show the following Algebraic Stability Theorem. 

\begin{theorem}[{\cite{ChaChoGli+09}}] \label{theorem:interlievedImpliesStability}
	Let $\mathcal{V} = \big(\{V^\delta\},\{\nu_\delta^{\delta'}\}\big)_{\delta\le\delta'\in \mathbb{R}}$ and $\mathcal{W} = \big(\{W^\delta\},\{\mu_\delta^{\delta'}\}\big)_{\delta\le\delta'\in \mathbb{R}}$ be two persistence vector spaces. Then,
	\[d_B\big(\Dgm(\mathcal{V}),\Dgm(\mathcal{W})\big) \le d_I(\mathcal{V},\mathcal{W}).\]
\end{theorem}

%We remark that the interleaving distance can be reformulated in terms of morphisms of persistence modules, as shown in \cite{BauLes14}.

\subsection{Dissimilarity functions and the correspondence distortion distance}\label{section:introAsymmetricNetrowks}

Our objective is to define a theory of persistent homology able to detect directed cycles modulo boundaries. Therefore, instead of building filtrations (of simplicial complexes) from finite metric spaces, we are interested in filtrations (of directed simplicial complexes) built from arbitrary dissimilarity functions. We will follow \cite{Tur19}, where they receive the name of set-function pairs, and \cite{CarMemRibSeg14,ChoMem18,ChoMem18b}, where they are referred to as dissimilarity networks or asymmetric networks. 

In this section, we introduce the necessary background on dissimilarity functions, and the correspondence distortion distance between them, which is a generalisation of the Gromov-Hausdorff distance to the asymmetric setting. We finish with a reformulation of this distance, established in \cite{CarMemRibSeg14}, which we will need to prove our persistent homology stability result.

\begin{definition}\label{definition:network}
	Let $V$ be a finite set. A \emph{dissimilarity function} $(V,d_V)$ on $V$ is a function $d_V\colon V\times V\to\mathbb{R}$. The value of $d_V$ on a pair $(v_1,v_2)$ is called the \emph{distance}, or \emph{dissimilarity}, \emph{from} $v_1$ \emph{to} $v_2$. 
\end{definition}

Note that no restrictions are imposed on $d_V$, thus it may not be symmetric, the triangle inequality may not hold, and the distance from a point to itself may not be zero.  These functions are also referred to as \emph{asymmetric networks} \cite{CarMemRibSeg14,ChoMem18,ChoMem18b} since they may be represented as a network with vertex set $V$ and an edge from $v_1$ to $v_2$ with weight $d_V(v_1,v_2)$ for every $(v_1,v_2)\in V\times V$. Thus, dissimilarity functions are very flexible and allow for the modelling of widely different problems, including any problem that can be modelled using a network. %For example, they can be used to model transport problems, where distances greater than $0$ from a vertex to itself codify the cost of storing the goods to be transported in the facility modelled by that vertex.

We would like to build (directed) simplicial complexes and, ultimately, persistence diagrams from dissimilarity functions. In order to check the stability of our constructions, we need a way to measure how close two such objects are. When comparing networks with the same vertex sets, a natural choice is the $\ell^\infty$ norm. However, we are interested in comparing dissimilarity functions on different vertex sets. To that end, we consider the $\ell^\infty$ norm over all possible pairings (quantified by a binary relation) between vertex sets, following ideas similar to those behind the definition of the Gromov-Hausdorff distance. 

\begin{definition}\label{definition:networkDistance}
	Let $(V,d_V)$ and $(W,d_W)$ be two dissimilarity functions and let $\theta$ be a non-empty binary relation between $V$ and $W$, that is, an arbitrary subset $\theta \subseteq V \times W$. The \emph{distortion} of the relation $\theta$ is defined as
	\[\dis(\theta) := \max_{(v_1,w_1), (v_2,w_2)\in \theta} |d_V(v_1,v_2) - d_W(w_1,w_2)|.\]
	A \emph{correspondence} between $V$ and $W$ is a relation $\theta$ between these sets such that $\pi_V(\theta)=V$ and $\pi_W(\theta) = W$, where $\pi_V \colon V \times W \to V$ is the projection onto $V$, and similarly for $\pi_W$. That is, $\theta$ is a correspondence if every element of $V$ is related to at least an element of $W$, and vice-versa. The set of all correspondences between $V$ and $W$ is denoted $\mathcal{R}(V,W)$.

	The \emph{correspondence distortion distance} \cite{Tur19} between $(V,d_V)$ and $(W,d_W)$ is defined as
	\[d_{\CD}\big((V,d_V), (W,d_W)\big) = \frac{1}{2}\min_{\theta\in\mathcal{R}(V,W)}\dis(\theta).\]
	This notion agrees with that of the Gromov-Hausdorff distance when $(V,d_V)$ and $(W,d_W)$ are metric spaces, \cite[Section 7.3.3]{BurBurIva01}.
\end{definition}

We will use a reformulation of the correspondence distortion that can be found in \cite{ChoMem18}. In order to introduce it, we need to define the distortion and co-distortion of maps between sets endowed with dissimilarity functions.

\begin{definition}
	Let $(V,d_V)$ and $(W,d_W)$ be any two dissimilarity functions and let $\varphi\colon V\to W$ and $\psi\colon W\to V$ be maps of sets. The \emph{distortion} of $\varphi$ (with respect to $d_V$ and $d_W$) is defined as
	\[\dis(\varphi) := \max_{v_1,v_2\in V} \big|d_V(v_1,v_2) - d_W\big(\varphi(v_1),\varphi(v_2)\big)\big|.\]
	The \emph{co-distortion} of $\varphi$ and $\psi$ (with respect to $d_V$ and $d_W$) is defined as
	\[\codis(\varphi,\psi) := \max_{(v,w)\in V\times W} \big|d_V\big(v,\psi(w)\big) - d_W\big(\varphi(v),w\big)\big|.\]
	Note that codistortion is not necessarily symmetrical, namely, $\codis(\varphi,\psi)$ and $\codis(\psi,\varphi)$ may be different if either of the dissimilarity functions are asymmetric.
\end{definition}

Finally, we have the following reformulation of the correspondence distortion distance.

\begin{proposition}[{\cite[Proposition 9]{ChoMem18}}]\label{proposition:networkDistance}
	Let $(V, d_V)$ and $(W, d_W)$ be any two dissimilarity functions. Then,
	\[d_{\CD}\big((V,d_V),(W,d_W)\big) = \frac{1}{2}\min_{\substack{\varphi\colon V\to W,\\ \psi\colon W\to V}}\big\{\max\left\{\dis(\varphi),\dis(\psi), \codis(\varphi,\psi), \codis(\psi,\varphi)\right\}\big\}.\]
\end{proposition}

\section{Directed homology of directed simplicial complexes} \label{section:homologySimplicialComplexes}

In this section, we introduce and study a family of subrings of the homology rings of simplicial complexes which encode directionality information and in particular detect directed cycles (Fig.~\ref{figure:directedCycles}). We cannot introduce such a family using (undirected) simplicial complexes, as they cannot encode directionality information of the simplices. One approach is to use the so-called \emph{ordered-set complexes}, where simplices are sets with a total order on the vertices. They generalise simplicial complexes, which can be encoded as fully symmetric ordered-set complexes, that is, ordered-set complexes where if a set of vertices forms a simplex, it must do so with every possible order. However, persistent homology of ordered-set complexes is not stable (Remark \ref{remark:stabilityAndRepetitions}). 

To achieve stability, we use one further generalisation, called \emph{ordered tuple complexes} or \emph{OT-complexes} in \cite{Tur19}, and called \emph{directed simplicial complexes} in this article (Definition \ref{definition:directedSimplicialComplex}). The only difference is that arbitrary repetitions of vertices are allowed in the ordered tuples representing simplices. Clearly, any ordered-set complex is a directed simplicial complex. Furthermore, a (undirected) simplicial complex can be encoded as a fully symmetric (as above) directed simplicial complex where every possible repetition of vertices is also included. When doing so, the (undirected) simplicial complex and its associated directed simplicial complex have isomorphic homologies over rings (Remark \ref{remark:orderedTupleAssociated}), and morphisms of simplicial complexes can be lifted to morphisms between their associated directed simplicial complexes (Remark \ref{remark:definitionMorphism}). Finally, and crucially for us, we will be able to use the above mentioned subrings encoding information on directed cycles to introduce a theory of persistent homology for dissimilarity functions which is stable with respect to the correspondence-distortion distance (Section \ref{section:directedPersistentHomology}). Note that, as ordered-set complexes are directed simplicial complexes, the results stated here apply to homology computations on ordered-set complexes as well.

\subsection{Chain complexes of directed simplicial complexes}

In this section, we introduce directed simplicial complexes and their associated chain complexes and homology modules, including modules of directed homology.  

\begin{definition}\label{definition:directedSimplicialComplex}
	A \emph{directed simplicial complex}, or \emph{ordered tuple complex}, is a pair $(V,X)$ where $V$ is a finite set of \emph{vertices}, and $X$ is a family of tuples $(x_0,x_1,\dots,x_n)$ of elements of $V$ such that if $(x_0,x_1,\dots,x_n)\in X$, then $(x_0,x_1,\dots,\widehat{x_i},\dots,x_n)\in X$ for every $i=0,1,\dots,n$. (Here and throughout the rest of the paper, $\widehat{x_i}$ indicates that $x_i$ is omitted from the list.) 
	
	We will denote the directed simplicial complex $(V,X)$ just by $X$, and assume that every vertex belongs to at least one directed simplex (so that $V$ is uniquely determined from $X$). Elements of $X$ of length $n+1$ are called \emph{$n$-simplices}, and the subset of the $n$-simplices of $X$ is denoted by $X_n$. Note that arbitrary repetitions of vertices in a tuple are allowed.
\end{definition}

Note that directed simplicial complexes are more general objects than simplicial complexes, in a way that allows directed simplicial complexes to encode asymmetric information. Note further that it is possible to define a geometric realisation for these objects in which simplices with vertex repetitions are collapsed to a simplex with no repetitions, thus simplices with repeated vertices play a role similar to that of the degenerate simplices in simplicial sets. %However, the information on the directionality of the simplices is lost in this process. 

We make use of the following terminology when dealing with directed simplicial complexes.  Elements of $X$ of length $n+1$ are called \emph{$n$-simplices}, and the subset of the $n$-simplices of $X$ is denoted by $X_n$. An $n$-simplex obtained by removing vertices from an $m$-simplex, $n \le m$, is said to be a \emph{face} of the $m$-simplex. Such face is called \emph{proper} if $n<m$, or equivalently, if at least vertex of the $m$-simplex is removed. A directed simplicial complex is said to be $n$-dimensional, denoted $\dim(X)=n$, if $X_{n+1}$ is trivial (empty) but $X_n$ is not. A collection $Y$ of simplices of $X$ that is itself a directed simplicial complex is said to be a \emph{directed simplicial subcomplex} (or, simply, \emph{subcomplex}) of $X$, denoted $Y\subseteq X$. Note that the vertex set of $Y$ may be strictly smaller than that of $X$.

We now introduce chain complexes associated to a directed simplicial complex. These definitions are straight generalisations of those for (undirected) simplicial complexes.

\begin{definition}
	Let $R$ be a commutative ring. The \emph{$n$-dimensional chains} of $X$ are defined as the elements of the free $R$-module generated by the $n$-simplices,
	\[C_n(X,R) = R\big(\left\{[x_0,x_1,\dots,x_n]\mid
	(x_0,x_1,\dots,x_n)\in X\right\}\big).\]
	We call the elements $[x_0,x_1,\dots,x_n]\in C_n(X,R)$, $(x_0,x_1,\dots,x_n)\in X$, \emph{elementary $n$-chains}.
\end{definition}

\begin{definition}\label{definition:chainComplexDirectedComplex}
	Let $X$ be a directed simplicial complex. For each $n>0$, we define a morphism of $R$-modules $\partial_n\colon C_n(X,R)\to C_{n-1}(X,R)$ by
	\[\partial_n([x_0,x_1,\dots,x_n]) = \sum_{i=0}^{n}(-1)^i [x_0,x_1,\dots,\widehat{x_{i}},\dots,x_n]\]
    For $n=0$, let $\partial_0\colon C_0(X,R)\to\{0\}$ be the trivial map.
\end{definition}

A straightforward computation analogous to the one for chain complexes in simplicial homology shows that $\{C_n(X,R),\partial_n\}$ is a chain complex of $R$-modules. In particular, we can define homology groups.

\begin{definition}\label{definition:directedHomologyComplexes}
	Let $X$ be a directed simplicial complex, $R$ a commutative ring, and $n \ge 0$. The $n$-dimensional \emph{cycles}, \emph{boundaries}, and \emph{homology} of $X$ are $Z_n(X,R) = \{x\in C_n(X,R)\mid \partial_n(x)=0\}$, $B_n(X,R)=\Im \partial_{n+1}$ and the quotient $H_n(X,R)=Z_n(X,R)/B_n(X,R)$, respectively.
\end{definition}

\begin{remark}\label{remark:orderedTupleAssociated}
	If $X$ is an (undirected) simplicial complex, we can define an ordered-set simplicial complex $X^{\text{OT}}$ where $(x_0,x_1,\dots,x_n)\in X^{\text{OT}}$ whenever $\{x_0,x_1,\dots,x_n\}$, after removing any repetition of vertices, is a simplex in $X$. The chain complex $C_*(X^{\text{OT}},R)$ is called the \emph{ordered chain complex} of $X$ in \cite{Mun84}, and its homology is isomorphic to the singular homology of $X$ over $R$.
\end{remark}

We now discuss how to only retain the information associated to directed cycles. Note that when introducing the differential for the chain complexes of simplicial complexes, an orientation is chosen for each of the simplices. However, undirected cycles can still be produced because the chosen orientation for a simplex can be reverted by changing the sign of the corresponding elementary chain. As such, homology of (undirected) simplicial complexes is independent of these choices of orientations.

In the case of directed simplicial complexes, orientations are not chosen, and we can ensure that we are only retaining information regarding directed cycles by making sure that all of the chains have positive coefficients, when the coefficient ring is chosen appropriately. However, we also require the introduced object to be a submodule of the entire module of homology. Thus, we arrive at the following definition.

\begin{definition}\label{definition:directed-homology}
	Let $X$ be a directed simplicial complex and let $R\in\{\mathbb{Z},\mathbb{Q},\mathbb{R}\}$. The set of \emph{$n$-dimensional directed cycles} of $X$ over $R$ is
	\[Z_n^{\Dir}(X,R) = \left\{\sum_i\lambda_i c_i\in Z_n(X,R)\mid \lambda_i\ge 0, \text{$c_i$ elementary $n$-chain}\right\}.\]
	The $n$th dimensional module of directed homology, $H_n^{\Dir}(X,R)$, is the submodule of $H_n(X,R)$ generated by the directed cycles,
	\[H_n^{\Dir}(X,R) = \left\langle\left\{[x]\mid x\in Z_n^{\Dir}(X,R)\right\}\right\rangle.\]
\end{definition}

\begin{remark}
    The rings $\mathbb{Z},\mathbb{Q}$ and $\mathbb{R}$ are used to simplify the exposition: the definition of $Z_n^{\Dir}(X,R)$, and thus of $H_n^{\Dir}(X,R)$, can be made in an analogous way for any ring $R$ that is the Grothendieck completion of a cancellative, zerosumfree semiring $\Lambda$ (see Appendix \ref{appendix:algebraicBackground}).
    %The object $Z_n^{\Dir}(X,R)$ is the semimodule of cycles of $X$ in homology with coefficients on the semiring $\Lambda$, where $\Lambda$ is the semiring of naturals, non-negative rationals or non-negative reals, depending on the choice of $R$. In fact, we would be successful in the detection of directed cycles by analogously defining $Z_n^{\Dir}(X, R)$ for any ring $R$ that is the Grothendieck completion of a cancellative, zerosumfree semiring $\Lambda$, see Appendix \ref{appendix:algebraicBackground} for more details. The rings $\mathbb{Z},\mathbb{Q}$ and $\mathbb{R}$ are used to simplify the exposition.
\end{remark}

In view of Definition \ref{definition:directed-homology} it is clear that, in dimension $1$, directed circuits in the directed graph underlying a directed simplicial complex will give rise to directed cycles. In Section \ref{section:computationsAndExamples} we will show how these directed submodules of homology behave in a desirable way.

\subsection{Functoriality of directed homology}\label{section:homologyFunctoriality}

In this section, we prove that homology and directed homology are functors from the category of directed simplicial complexes to the category of graded $R$-modules. We also show that two morphisms allowing for the construction of the prism operator must induce the same map on homology, a result we will need to prove our persistent homology stability result. We begin by introducing morphisms of directed simplicial complexes.

\begin{definition}\label{definition:simplicialMorphism}
	A \emph{morphism of directed simplicial complexes} $f\colon (V,X)\to (W,Y)$ or just $f\colon X\to Y$, is a map of sets $f\colon V\to W$ such that if $(x_0,x_1,\dots,x_n)\in X$ then $\big(f(x_0), f(x_1),\dots, f(x_n)\big)\in Y$.
\end{definition}

\begin{remark}\label{remark:definitionMorphism}
    This definition is stricter than the classical notion of morphism of simplicial complexes, where morphisms are allowed to take simplices to simplices of lower dimension, as directed simplicial complexes can intrinsically account for vertex repetitions. However, note that if $X$ and $Y$ are (undirected) simplicial complexes, a map $f\colon X\to Y$ is a morphism of simplicial complexes if and only if $f\colon X^{\text{OT}}\to Y^{\text{OT}}$ (see Remark \ref{remark:orderedTupleAssociated}) is a morphism of directed simplicial complexes.
\end{remark}

\begin{definition}
	Let $f\colon X\to Y$ be a morphism of directed simplicial complexes. Then $f$ induces morphisms of $R$-modules $C(f) = \{C_n(f)\}$, defined as
	\begin{align*} C_n(f)\colon C_n(X,R)&\longrightarrow C_n(Y,R) \\
		[x_0,x_1,\dots,x_n]&\longmapsto [f(x_0), f(x_1),\dots,f(x_n)].
	\end{align*}
	We will often abbreviate $C_n(f)=f_n$.
\end{definition}

\begin{proposition}\label{proposition:degmorphism}
	If $f\colon X\to Y$ is a morphism of directed simplicial complexes, the family of maps $\{C_n(f)\}$ is a $R$-homomorphism of chain complexes. Therefore, it induces a map $H_n(f)\colon H_n(X,R)\to H_n(Y,R)$. Furthermore, $H_n(f)$ restricts to a map $H_n^{\Dir}(f)\colon H_n^{\Dir}(X,R)\to H_n^{\Dir}(Y,R)$.
\end{proposition}

\begin{proof}
    Let $[x_0,x_1,\dots,x_n]\in C_n(X,R)$ be a simplex. Simple computations show that $C_{n-1}(f)\partial_n = \partial_n C_n(f)$, thus inducing a map $H_n(f)\colon H_n(X,R)\to H_n(Y,R)$.
    
    Now let $x\in Z_n^{\Dir}(X,R)$. We can write $x = \sum_i \lambda_i x_i$, where $\lambda_i\ge 0$ and $x_i$ is an elementary $n$-chain. Then, $C_n(f)(x) = \sum_i\lambda_i C_n(f)(x_i)$. Clearly, $C_n(f)$ takes elementary $n$-chains to elementary $n$-chains, thus $C_n(f)(x)\in Z_n^{\Dir}(Y,R)$. It is then clear from the definition of directed homology that $H_n(f)$ restricts to $H_n(f)\colon H_n(X,R)\to H_n(Y,R)$.
\end{proof}

\begin{remark}\label{remark:morfismOTComplex}
	Let $X$ and $Y$ be (undirected) simplicial complexes and let $R$ be a ring. If $f\colon X\to Y$ is a morphism of (undirected) simplicial complexes, the map induced on homology by $f\colon X\to Y$ is the same as the map induced on homology by $f\colon X^{\text{OT}}\to Y^{\text{OT}}$ (see Remarks \ref{remark:orderedTupleAssociated} and \ref{remark:definitionMorphism}).
\end{remark}

\begin{corollary}
	(Directed) Homology is a functor from the category of directed simplicial complexes to the category of graded $R$-modules. In particular, isomorphic directed simplicial complexes have isomorphic (directed) homologies.	
\end{corollary}

We finish this section by showing a sufficient condition for two morphisms to induce the same map on homology. We will need this result to prove that our definition of persistent homology is stable.

\begin{lemma}\label{lemma:prismConstruction}
	Let $R$ be a commutative ring. Let $X$ and $Y$ be two directed simplicial complexes and let $f,g\colon X\to Y$ be morphisms of directed simplicial complexes such that if $(x_0,x_1,\dots,x_n)\in X$, then $\big(f(x_0),f(x_1),\dots,f(x_i),g(x_i),\dots,g(x_n)\big)\in Y$ for every $i=0,1,\dots,n$. Then, $H_n(f) = H_n(g)$ for every $n \ge 0$. Consequently, $H_n^{\Dir}(f)=H_n^{\Dir}(g)$, for every $n\ge 0$.
\end{lemma}

\begin{proof}
	For $x = [x_0,x_1,\dots,x_n]\in C_n(X,R)$ an elementary $n$-chain, define
	\[
	s_n[x_0,x_1,\dots,x_n] = (-1)^i \sum_{i=0}^n [f(x_0),\dots,f(x_{i}),g(x_{i}),\dots,g(x_n)].\]
	We show that this family of maps provides a chain homotopy between $f$ and $g$. Namely, we prove that $C_n(g) - C_n(f) = \partial_{n+1} s_n + s_{n-1}\partial_n$, for all $n$. On the one hand,
	\begin{equation}\label{eq:sdx}
	    \begin{split}
	        s_{n-1}\partial_n(x) = \sum_{i=0}^n (-1)^i \left[ \sum_{j=0}^{i-1} (-1)^j \big[f(x_0), f(x_1),\dots, f(x_j), g(x_j),\dots, \widehat{g(x_i)},\dots, g(x_n)\big] \right.\\
	        \left.+ \sum_{j=i+1}^n (-1)^{j+1} \big[f(x_0),f(x_1),\dots, \widehat{f(x_i)},\dots, f(x_j),g(x_j),\dots, g(x_n)\big]\right].
	    \end{split}
	\end{equation}
	On the other hand,
    \begin{equation*}
        \begin{split}
            \partial_{n+1} s_n(x) = \sum_{j=0}^n (-1)^j\left[ \sum_{i=0}^j (-1)^i \big[f(x_0), f(x_1),\dots, \widehat{f(x_i)},\dots, f(x_j), g(x_j),\dots,g(x_n)\big] \right.\\
            \left.+ \sum_{i=j}^n (-1)^{i+1} \big[f(x_0),f(x_1),\dots,f(x_j),g(x_j),\dots,\widehat{g(x_i)},\dots,g(x_n)\big]\right].
        \end{split}
    \end{equation*}
    By exchanging the roles of the indices in this last equation, we obtain that
	\begin{equation}\label{eq:dsx}
	    \begin{split}
	        \partial_{n+1} s_n(x) = \sum_{i=0}^n \left[\sum_{j=0}^i (-1)^{i+1} (-1)^j \big[f(x_0),f(x_1),\dots,f(x_j),g(x_j),\dots,\widehat{g(x_i)},\dots,g(x_n)\big]\right.\\
	        \left.+ \sum_{j=i}^n (-1)^i (-1)^j\big[f(x_0),f(x_1),\dots, \widehat{f(x_i)},\dots, f(x_j),g(x_j),\dots, g(x_n)\big]\right].
	    \end{split}
	\end{equation}
    Now, adding Equations \eqref{eq:sdx} and \eqref{eq:dsx},
	\begin{align*}
	    \partial_{n+1} s_n(x) + s_{n-1}\partial_n(x) = &\sum_{i=0}^n [f(x_0), f(x_1),\dots,f(x_{i-1}),g(x_i),\dots,g(x_n)] \\ 
	    - &\sum_{i=0}^n \big[f(x_0),f(x_1),\dots,f(x_i),g(x_{i+1}),\dots,g(x_n)\big]. 
	\end{align*}
	It is now straightforward to check that this sum amounts to
	\[[g(x_0), g(x_1),\dots, g(x_n)] - [f(x_0),f(x_1),\dots,f(x_n)] = C_n(g)(x) - C_n(f)(x),\]
	and we are done.
\end{proof}

\subsection{Basic computations and examples}\label{section:computationsAndExamples}

In this section, we explore some basic properties of this homology theory. We begin by studying the relation between homology and connectivity.

\begin{definition}
    Let $(V,X)$ be a directed simplicial complex and $v,w \in V$ be two vertices. A \emph{path} from $v$ to $w$ in $X$ a sequence of vertices $v=x_0,x_1,\ldots,x_n=w$ such that either  $(x_i,x_{i+1})$ or $(x_{i+1},x_{i})$ is a simplex, for all $1\le i \le n-1$. The \emph{(weakly) connected components} of $X$ are the equivalence classes of the equivalence relation $v \sim w$ if there is a path in $X$ from $v$ to $w$. We call $X$ \emph{(weakly) connected} if it has only one connected component, that is, every pair of vertices can be connected by a path. 
\end{definition}

Note that this notion of connectedness ignores the direction of the edges (1-simplices). The next result shows that we can compute the homology of each connected component independently. The proof follows by observing that the chain complex $C_n(X,R)$ is clearly the direct sum of the chain complexes associated to each of the (weakly) connected components of $X$. In particular, it is immediate that any directed cycle can be decomposed as a sum of directed cycles in each of the directed components.

\begin{proposition}\label{proposition:connectedComponents}
    The (directed) homology of a directed simplicial complex $X$ is isomorphic to the direct sum of the (directed) homology modules of its (weakly) connected components.
\end{proposition}

We next show that the $0$th (directed) homology counts the number of (weakly) connected components of a directed simplicial complex $X$. We start with a lemma. 

\begin{lemma}\label{lemma:verticesneverboundaries}
	Let $R$ be a commutative ring and $X$ be a directed simplicial complex. For any vertex $v$ and any $\lambda\in R$, $\lambda[v]$ is homologous to the trivial class if and only if $\lambda = 0$.
\end{lemma}

\begin{proof}
    It is clear that $\lambda[v]$ is a cycle as $\partial_0$ is trivial. On the other hand, it is clear that it must be the trivial class if $\lambda=0$. Now assume that $\lambda[v]$ is the trivial class. We shall prove that $\lambda=0$.

	Assume that $X$ has $n+1$ vertices, namely $V = \{x_0 = v, x_1,\dots, x_n\}$. Then if $\lambda[x_0]$ is trivial, there exists a $1$-chain
	\[x = \sum_{i,j=0}^n \lambda_{ij} [x_i, x_j]\]
	such that $\lambda[x_0] = \partial(x)$, where we are assuming that $\lambda_{ij} = 0$ whenever $[x_i,x_j]$ is not a $1$-chain. By computing the differentials,
		\[\lambda[x_0] = \sum_{i,j=0}^n \lambda_{ij} [x_j] - \sum_{i,j=0}^n \lambda_{ij} [x_i].\]
	Now, since  $\big\{[x_0], [x_1], \dots, [x_n]\big\}$ is a basis of $C_0(X,R)$, the coefficients of each element in the basis must be equal in both sides of the equality. Namely, for each $j \neq 0$, $0 = \sum_{i=0}^n (\lambda_{ij} - \lambda_{ji})$, and $\lambda =  \sum_{i=0}^n (\lambda_{i0} - \lambda_{0i})$. Adding these equations, we deduce that $\lambda = 0$.
\end{proof}

\begin{proposition}\label{proposition:H0andConnectivity}
    Let $X$ be a directed simplicial complex and let $R$ be a commutative ring. Then $H_0(X,R) \cong R^k$, where $k$ is the number of (weakly) connected components of $X$. Furthermore, if $R\in\{\mathbb{Z},\mathbb{Q},\mathbb{R}\}$, then $H_0(X,R) = H_0^{\Dir}(X,R)\cong R^k$.
\end{proposition}

\begin{proof}
    By Proposition \ref{proposition:connectedComponents} we only need to show that if $X$ is connected, then $H_0(X,R)\cong R$. Let $v$, $w$ be any two vertices and let us show that $[v]$ and $[w]$ are homologous $0$-cycles. Indeed, since $X$ is connected, we can find vertices $v = v_0, v_1,\dots, v_m = w$ such that either $[v_i, v_{i+1}]$ or $[v_{i+1},v_i]$ are $1$-chains, for all $i = 0,1,\dots,n-1$. If either $[v_i, v_{i+1}]$ or $[v_{i+1}, v_i]$ are chains, then $[v_i]$ and $[v_{i+1}]$ are homologous. As this holds for every $i = 0,1,\dots,{m-1}$, $[v]$ is homologous to $[w]$. Now, if $\sum_{i=1}^n\lambda_i [x_i]$ is any $0$-cycle, it is homologous to $\sum_{i=1}^n \lambda_i [v]$. Then, by Lemma \ref{lemma:verticesneverboundaries}, $\lambda[v]$ is not homologous to $\mu[v]$ whenever $\lambda\ne \mu$, thus $H_0(X,R) \cong R$.
    
    For the second statement, if $R\in\{\mathbb{Z},\mathbb{Q},\mathbb{R}\}$, we can compute $H_n^{\Dir}(X,R)$. In such case, $[v]$ is a directed cycle, thus $\langle [v]\rangle = H_0(X,R)\subseteq H_0^{\Dir}(X,R)$. It then follows that $H_0(X,R) = H_0^{\Dir}(X,R)$. 
\end{proof}

\begin{remark}
The fact that 0th directed homology ignores edge directions seems counter-intuitive, as we set up this framework to detect directed features, namely directed cycles. However, this is a natural consequence of homology being an equivalence relation. In particular, as mentioned in the proof, if either $[u,v]$ or $[v,u]$ are $1$-chains, $[u]$ and $[v]$ are homologous, thus directionality of the $1$-chains does not affect 0th directed homology. However, this only occurs in dimension $0$, as a consequence of $0$-chains always being cycles. Thus, when working with 1-simplices, the symmetry in the homology relation does not affect the detection of asymmetry in the data, as such asymmetry is encoded in the cycles themselves. It is also worth mentioning that a persistent homology able to detect strongly connected components in directed graphs has been introduced in \cite{Tur19}.
\end{remark}

We now show that the homology of a point is trivial for positive indices.

\begin{example}\label{example:pointAcyclic}
	Let $R$ be a commutative ring and let $X$ be the directed simplicial complex with vertex set $V = \{x_0\}$ and simplices $X=\{(x_0)\}$. Then,
	\[H_n(X,R) = \begin{cases}
		R, & \text{if $n = 0$},\\
		0, & \text{if $n > 0$.}
	\end{cases}\]
	Indeed, $X$ is connected, thus $H_0(X,R) = R$ by Proposition \ref{proposition:H0andConnectivity}. For $n\ge 1$, there are no $n$-chains, hence $H_n(X,R)=0$. In this case, it is also immediate to check that if $R\in\{\mathbb{Z},\mathbb{Q},\mathbb{R}\}$, $H_n^{\Dir}(X,R) = H_n(X,R)$, for every $n$.
\end{example}

Directed simplicial complexes whose homology is isomorphic to that of the point are called \emph{acyclic}. 

\begin{definition}\label{definition:acyclic}
	A directed simplicial complex $X$ is \emph{acyclic} if
		\[H_n(X,R) = \begin{cases}
		R, & \text{if $n = 0$},\\
		0, & \text{if $n > 0$.}
	\end{cases}\]
	If $R\in\{\mathbb{Z},\mathbb{Q},\mathbb{R}\}$, we say that $X$ is \emph{directionally acyclic} if
	\[H_n^{\Dir}(X,R) = \begin{cases}
		R, & \text{if $n = 0$},\\
		0, & \text{if $n > 0$.}
	\end{cases}\]
	Note that acyclic implies directionally acyclic, as directed homology is a submodule of the module of homology, for every $n$.
\end{definition}

Next, we show that a directed simplicial complex consisting of one $m$-dimensional simplex along with all of its faces (note that this is an ordered-set complex) is acyclic.

\begin{proposition}\label{proposition:simplicesAcyclic}
    Let $X$ be the directed simplicial complex consisting of the simplex $(x_0,x_1,\dots,x_m)$ and all of its faces. Then $X$ is acyclic.
\end{proposition}

\begin{proof}
    As a consequence of how $X$ is defined, if $x\in C_n(X,R)$ is a chain so that $x_0$ does not participate in any of its elementary $n$-chains (as there are no repetitions in the simplices of $X$), then $x_0$ can be added as the first element of every elementary $n$-chain in $x$, giving us an $(n+1)$-chain which we denote $x_0 x\in C_{n+1}(X,R)$. Simple computations show that
    \[\partial(x_0 x) = x - x_0\partial(x).\]
    
    Now take $x \in Z_n(X,R)$ any cycle and decompose it as $x = x_0 y + z$, where $x_0$ does not participate in any of the chains in either $y$ or $z$. Using the formula above,
    \[\partial(x) = \partial(x_0 y + z) = y - x_0\partial(y) + \partial(z) = 0.\]
    In particular, the chains in which $x_0$ does not participate must add up to zero, namely $y + \partial(z) = 0$. Finally, consider the chain $x_0 z$. We have that
    \[\partial(x_0 z) = z - x_0\partial(z) = z + x_0 y = x.\]
    Thus, $x \in Z_n(X,R)$ is homologous to zero, and the result follows.
\end{proof}

We now illustrate the ability of this homology theory to detect directed cycles. We do so through some simple examples.

\begin{example}\label{example:twoTriangles}
    Let $R \in \{\mathbb{Z}, \mathbb{Q}, \mathbb{R} \}$. Let $X$, $Y$ and $Z$ be the directed simplicial complexes represented in Figure \ref{figure:triangles}. These three complexes are connected, so their $0$th homologies are $R$, both directed and undirected. Furthermore, neither $X$ nor $Y$ have $k$-simplices for $k\ge 2$, so $H_k(X,R)$ and $H_k(Y,R)$ are trivial (zero) for every $k\ge 2$. This in particular implies that $H_k^{\Dir}(X,R)$ and $H_k^{\Dir}(Y,R)$ are trivial as well. Although $Z$ has one $2$-simplex, it is not a $2$-cycle, thus $H_2(Z,R)$ is also trivial.
	\begin{figure}
		\begin{center}
			\begin{tikzpicture}
				\tikzset{node distance = 1.5cm, auto}
				\node[style=point,label={[label distance=0cm]270:$v_1$}](v1) {};
				\node[style=point, right of = v1, label={[label distance = 0cm]270:$v_2$}](v2) {};
				\node[style=point, above of = v1, xshift=0.75cm, yshift = -0.2 cm, label={[label distance = 0cm]90:$v_3$}](v3) {};
				\draw[->, thick, >=stealth] (v1) to node{} (v2);
				\draw[->, thick, >=stealth] (v1) to node{} (v3);
				\draw[->, thick, >=stealth] (v2) to node{} (v3);
				\node[below of = v3, yshift=-0.5cm] {$X$};
%%%%%%%%%%%%%%%%%%%%%%%%%%%%%%%%%%%%%%%%%%%%%%%%
				\node[style=point, right of = v2, xshift = 1cm, label={[label distance = 0cm]270:$w_1$}](w1) {};
				\node[style=point, right of = w1, label={[label distance = 0cm]270:$w_2$}](w2) {};
				\node[style=point, above of = w1, xshift=0.75cm, yshift = -0.2 cm, label={[label distance = 0cm]90:$w_3$}](w3) {};
				\draw[->, thick, >=stealth] (w1) to node{} (w2);
				\draw[->, thick, >=stealth] (w3) to node{} (w1);
				\draw[->, thick, >=stealth] (w2) to node{} (w3);
				\node[below of = w3, yshift=-0.5cm] {$Y$};
%%%%%%%%%%%%%%%%%%%%%%%%%%%%%%%%%%%%%%%%%%%%%%%%
    			\node[style=point,label={[label distance=0cm]270:$u_1$},right of = w2,xshift=1cm](u1) {};
    			\node[style=point, right of = u1, label={[label distance = 0cm]270:$u_2$}](u2) {};
    			\node[style=point, above of = u2, xshift=0.75cm, yshift = -0.2 cm, label={[label distance = 0cm]90:$u_3$}](u3) {};
                \node[style=point, left of = u3, label={[label distance = 0cm]90:$u_4$}] (u4) {};
                \draw[->,thick,>=stealth] (u1) to node {} (u2);
                \draw[->,thick,>=stealth] (u2) to node {} (u3);
                \draw[->,thick,>=stealth] (u3) to node {} (u4);
                \draw[->,thick,>=stealth] (u4) to node {} (u1);
                \draw[->,thick,>=stealth] (u2) to node {} (u4);
                \begin{pgfonlayer}{background}
                \path[fill=gray!50,opacity=.7] (u3.center) to (u2.center) to (u4.center);				\node[below of = u4, yshift=-0.5cm] {$Z$};
                \end{pgfonlayer}
			\end{tikzpicture}
			\caption{Directed homology detects directed $1$-cycles modulo boundaries. In these examples, $H_1^{\Dir}(X, R)=0$ while $H_1^{\Dir}(Y,R)=H_1^{\Dir}(Z,R)=R$ respectively generated by $[w_1,w_2]+[w_2,w_3]+[w_3,w_1]$ and $[u_1,u_2]+[u_2,u_4]+[u_4,u_1]$ (or $[u_1,u_2]+[u_2,u_3]+[u_3,u_4]+[u_4,u_1]$). To be a directed 1-cycle, the directions of the involved $1$-simplices must form circuit. This does not hold for (undirected) homology.}
			%The final example shows that $2$-simplices can make $1$-simplices homologous.}
			\label{figure:triangles}
		\end{center}
	\end{figure}
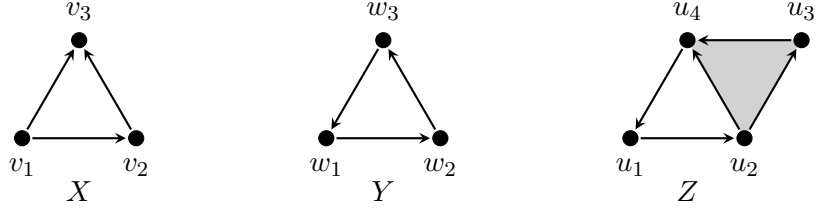
	To compute the first homology, we need to consider the $1$-simplices and their differentials. We list them below.
    \begin{center}
        \begin{tabular}{c|ccc|ccc|c}
            $X$ & $\partial$ & \hspace{1cm} & $Y$ & $\partial$ &  \hspace{1cm} & $Z$ & $\partial$ \\\cline{1-2}\cline{4-5}\cline{7-8}
            $[v_1,v_2]$ & $v_2-v_1$ & & $[w_1,w_2]$ & $w_2-w_1$ & & $[u_1,u_2]$ & $u_2-u_1$\\\cline{1-2}\cline{4-5}\cline{7-8}
            $[v_2,v_3]$ & $v_3-v_2$ & & $[w_2,w_3]$ & $w_3-w_2$ & & $[u_2,u_3]$ & $u_3-u_2$\\\cline{1-2}\cline{4-5}\cline{7-8}
            $[v_1,v_3]$ & $v_3-v_1$ & & $[w_3,w_1]$ & $w_1-w_3$ & & $[u_3,u_4]$ & $u_4-u_3$\\\cline{7-8}
            \multicolumn{6}{c}{} & $[u_4,u_1]$ & $u_1-u_4$\\\cline{7-8}
             \multicolumn{6}{c}{} & $[u_2,u_4]$ & $u_4-u_2$
        \end{tabular}
    \end{center}
    We start with $X$. Note that $\lambda_1 [v_1,v_2] + \lambda_2[v_2,v_3]+\lambda_3[v_1,v_3]$, $\lambda_1,\lambda_2,\lambda_3\in R$, is a cycle for X if and only if $\lambda_1 = \lambda_2 = - \lambda_3$. Since there are no $2$-cycles in $X$, different $1$-cycles cannot be homologous. We deduce that $H_1(X,R) \cong R$. However, no $1$-cycle can have non-negative coefficients for the three elementary $1$-cycles, thus $H_1^{\Dir}(X,R) = \{0\}$.
    
    Regarding $Y$, note that $\mu_1[w_1,w_2] + \mu_2[w_2,w_3] + \mu_3[w_3,w_1]$, $\mu_1,\mu_2,\mu_3\in R$, is a cycle if and only if $\mu_1 = \mu_2 = \mu_3$. Since there are no $2$-simplices in $Y$, we deduce that $H_1(Y,R) = R$. Furthermore, the coefficients can be simultaneously positive. Namely, $[w_1,w_2] + [w_2,w_3] + [w_3,w_1]$ is a directed cycle, which generates the entirety of $H_1(X,R)$. We deduce that $H_1^{\Dir}(X,R)\cong R$. Note that $X$ and $Y$ have isomorphic homology rings, but their directed homology rings are different.
     
    Finally, for $Z$, similar computations show that $Z_1(Z,R)$ is the free $R$-module generated by $x_1 = [u_1,u_2] + [u_2,u_3] + [u_3,u_4] + [u_4, u_1]$ and $x_2 = [u_1,u_2] + [u_2,u_4] + [u_4, u_1]$, both of which are directed cycles. In particular, $Z_1^{\Dir}(Z,R)$ consists of linear combinations of those two cycles, with positive coefficients. Note, however, that generating sets of $Z_1(X,R)$ containing fewer than two directed cycles exist. This shows that computing a generating set for the cycles may not be enough to compute the full set of directed cycles. 
    
    Continuing with the homology computations, note that in this case we have a $2$-simplex, $y=[u_2,u_3,u_4]$, for which $\partial(y) = [u_3,u_4] - [u_2,u_4] + [u_2, u_3] = x_1 - x_2$. Namely, $x_1$ and $x_2$ are homologous. Therefore, $H_1(Z,R) = H_1^{\Dir}(Z,R) = R$, generated by either $x_1$ or $x_2$. In particular, both the directed and undirected homologies of $Y$ and $Z$ are isomorphic and we can see how $2$-simplices can make directed cycles equivalent, as expected.
\end{example}

More generally, if $R\in \{\mathbb{Z},\mathbb{Q},\mathbb{R}\}$, a \emph{polygon} will give rise to a non-trivial directed homology class in dimension $1$ if, and only if, the cycle can be traversed following the direction of the edges. %Namely, only \emph{directed} cycles are detected in directed homology. 

\begin{proposition}\label{proposition:homologyPolygons}
    Let $X$ be a $1$-dimensional directed simplicial complex with vertices $v_0,v_1,\dots,v_n$, $n\ge 1$, and with $1$-simplices $e_0$, $e_1,\dots, e_n$ where either $e_i = (v_i, v_{i+1})$ or $e_i = (v_{i+1}, v_i)$, $i=0,1,\dots,n$, and we write $v_{n+1}=v_0$. Then, $H_1(X,R) = R$. If $R\in \{\mathbb{Z},\mathbb{Q},\mathbb{R}\}$, then $H_1^{\Dir}(X,R)$ is non-trivial if and only if either $e_i = (v_i,v_{i+1})$ for all $i$ or $e_i = (v_{i+1},v_i)$ for all $i$, in which case $H_1^{\Dir}(X,R)\cong R$. 
\end{proposition}

\begin{proof}
    We first prove that $H_1(X,R) = R$. Let $x = \sum_{i=0}^n \lambda_i [e_i]$ be a cycle. If $x$ is non-trivial, we can assume without loss of generality that $\lambda_0\ne 0$. We can further assume that $e_0 = (v_0, v_1)$. Thus, $\partial(\lambda_0 [e_0]) = \lambda_0[v_1] - \lambda_0[v_0]$. The only other edge in which $e_1$ participates is $e_1$. We deduce that $\lambda_1 = \lambda_0$, if $e_1 = (v_1, v_2)$, or $\lambda_1 = -\lambda_0$, if $e_1 = (v_2, v_1)$. By proceeding iteratively, we deduce that $\lambda_i = \lambda_0$, if $e_i = (v_i, v_{i+1})$, or that $\lambda_i = -\lambda_0$, if $e_i = (v_{i+1}, v_i)$. It is now straightforward to check that, with $\lambda_i$ defined in terms of $\lambda_0$ in such way, $i = 1,\dots,n$, $x = \sum_{i=0}^n \lambda_i [e_i]$ is a non-trivial cycle. Since there are no $2$-chains, $H_1(X,R) \cong R$.

    Regarding the directed homology, note that if $e_0 = (v_0, v_1)$, the coefficients of the cycle $x$ can only be all positive if and only if $e_i = (v_i, v_{i+1})$. In such case, $H_1^{\Dir}(X,R) = H_1(X,R) \cong R$, and otherwise $H_1^{\Dir}(X,R) = 0$. The case in which $e_0 = (v_1, v_0)$ is analogous.
\end{proof}

\begin{remark}
    Note that Proposition \ref{proposition:homologyPolygons} applies to polygons with only two vertices $v_1$ and $v_2$ (see Figure \ref{figure:twoVertexCycle}), which are allowed in a directed simplicial complex. Indeed, an immediate computation shows that $[v_1,v_2] + [v_2,v_1]$ is a directed 1-cycle. This result shows that directed homology is not a homotopy invariant for the geometric realisation of directed simplicial complexes, as all of the polygons considered in Proposition \ref{proposition:homologyPolygons} give rise to homotopic geometric realisations.
\end{remark}

\begin{figure}
    \begin{center}
        \tikzset{node distance = 1.5cm, auto}
        \begin{tikzpicture}
    			\node[style=point,label={[label distance=0cm]270:$v_1$}](v1) {};
    			\node[style=point, right of = v1, label={[label distance = 0cm]270:$v_2$}](v2) {};
                \draw[->,thick,>=stealth] (v1) to[out = 30, in = 150] node {} (v2);
    			\draw[->,thick,>=stealth] (v2) to[out = 210, in = -30] node {} (v1);
    	\end{tikzpicture}
        \caption{A directed simplicial complex with just two vertices can have a non-trivial $1$-cycle.}
        \label{figure:twoVertexCycle}
    \end{center}
\end{figure}
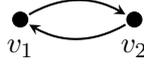

We end this section by noting that, although our initial aim was a directed homology group able to detect directed $1$-cycles, we have a directed homology theory defined in all dimensions. However, we can show that the directed homology is trivial in even dimensions other than 0.

\begin{proposition}\label{proposition:evenDimTrivialHomology}
    Let $X$ be a directed simplicial complex $R\in \{\mathbb{Z},\mathbb{Q},\mathbb{R}\}$. Then $H_{2n}^{\Dir}(X,R)=\{0\}$, for every $n\ge 1$.
\end{proposition}

\begin{proof}
    We will show that no non-trivial directed cycles may exist in such dimensions. In order to do so, consider the morphism of $R$-modules
    \begin{align*}
        \varphi_n\colon C_n(X,R)&\longrightarrow R\\
        [v_0,v_1,\dots,v_n]&\longmapsto 1.
    \end{align*}
    Thus, if $x = \sum_{i=0}^k \lambda_i x_i$ where $x_i$ is an elementary $n$-chain, $\varphi_n(x)=\sum_{i=0}^k \lambda_i$.
    
    Now assume that $x_i$ is an elementary $2n$-chain, for some $n\ge 1$. In this case, $\partial(\lambda_i x_i) = \lambda_i\partial(x_i)$, where $\partial(x_i)$ consists of the sum of $\lfloor\frac{2n}{2}\rfloor = n$ elementary $n$-chains with positive sign and $\lfloor\frac{2n-1}{2}\rfloor = n-1$ elementary $n$-chains with negative sign. Thus, $\varphi_{2n-1}(\lambda_i x_i) = \lambda_i (n - (n - 1)) = \lambda_i$.
    
    To prove our claim, let $x = \sum_{i=0}^k \lambda_i x_i\in Z_{2n}(X,R)$ be a cycle where each $x_i$ is an elementary $2n$-cycle. Then, given that $\partial(x) = 0$, $\varphi_{2n-1}\big(\partial(x)\big) = \sum_{i=0}^k \lambda_i = 0$. If $x$ is directed, then $\lambda_i \ge 0$, for all $i$, thus we deduce that $\lambda_i = 0$, for all $i$. Namely, the trivial cycle is the only directed cycle, and $H_{2n}^{\Dir}(X,R)=\{0\}$.
\end{proof}

\begin{remark}
    The proof of this result is based on the fact that the number of positive coefficients in the boundary of an elementary $n$-chain is larger than the number of negative coefficients, if $n$ is even. This may also be related to the fact that there is no obvious way to define directed $n$-cycles for $n>1$. For the purposes of this article, it suffices to consider 1-cycles and $H_1(X,R)$. However, note that non-trivial (homological) cycles do exist in all odd dimensions. For instance, the elementary $(2n-1)$-chain obtained by repeating a vertex $2n$ times is always a cycle.
\end{remark}

\section{Persistent directed homology}\label{section:directedPersistentHomology}

In this section, we introduce a theory of persistent homology for directed simplicial complexes which comes in two flavours: one takes into account the directionality of the complex, whereas the other one is analogous to persistent homology in the classical setting. We thus have, associated to the same filtration of directed simplicial complexes, two persistence modules which produce two different barcodes. Both persistent homology theories show stability (see Theorem \ref{theorem:stability}) and they are closely related. 
Indeed, directed cycles are undirected cycles as well, thus every bar in a directed persistence barcode can be uniquely matched with a bar in the corresponding undirected one, although the undirected bar may be born sooner, see Proposition \ref{proposition:subbarcode}. 

\subsection{Persistence modules associated to a directed simplicial complex}
Let us begin by introducing filtrations of directed simplicial complexes.

\begin{definition}
	Let $X$ be a directed simplicial complex (Definition \ref{definition:directedSimplicialComplex}). A \emph{filtration} of $X$ is a family of subcomplexes $(X_\delta)_{\delta\in T}$, $T\subseteq\mathbb{R}$, such that if $\delta\le \delta'\in T$, then $X_\delta$ is a subcomplex of $X_{\delta'}$, and such that $X = \cup_{\delta\in T} X_\delta$. Note that for $\delta\le \delta'$, the inclusion $i_\delta^{\delta'}\colon X_\delta\to X_\delta'$ is a morphism of directed simplicial complexes.	
\end{definition}

We now introduce undirected persistence modules.

\begin{definition}\label{definition:undirPersistenceModFilt}
	Let $(X_\delta)_{\delta\in T}$ be a filtration of a directed simplicial complex $X$ and let $R$ be a commutative ring. The $n$-dimensional \emph{undirected persistence $R$-module} of $X$ is the persistence $R$-module
        \[(\{H_n(X_\delta, R)\}, \{H_n(i_\delta^{\delta'})\}\big)_{\delta\le\delta'\in T}.\]
	The functoriality of $H_n$ makes this a persistence module.
\end{definition}

Now, in order to retain the information on directionality, we take the submodule of directed classes.

\begin{definition}\label{definition:dirPersistenceModFilt}
	Let $(X_\delta)_{\delta\in T}$ be a filtration of a directed simplicial complex $X$ and let $R\in \{\mathbb{Z},\mathbb{Q},\mathbb{R}\}$. The $n$-dimensional \emph{directed persistence $R$-module} of $X$ is the persistence $R$-module
	\[ \big( \{H_n^{\Dir}(X_\delta,R)\}, \{H_n^{\Dir}(i_\delta^{\delta'})\}\big)_{\delta\le\delta'\in T},\]
    where $H_n^{\Dir}(X_\delta,R)$ are defined as in Definition \ref{definition:directed-homology}, $i_\delta^{\delta'}\colon X_\delta\to X_{\delta'}$ are the inclusion maps, and $H_n^{\Dir}(i_\delta^{\delta'})$ are the restrictions of the maps $H_n(i_\delta^{\delta'})$ to directed homology, as introduced in Proposition \ref{proposition:degmorphism}. 
\end{definition}

\begin{remark}\label{remark:persistenceSubmodule}
    Let $(X_\delta)_{\delta\in T}$ be a filtration of a directed simplicial complex $X$ and let $\delta\le\delta'\in T$. If $R\in \{\mathbb{Z},\mathbb{Q},\mathbb{R}\}$, by applying Proposition \ref{proposition:degmorphism} to the map $i_\delta^{\delta'}\colon X_\delta\to X_{\delta'}$, we have a commutative diagram
    \begin{center}
    \begin{tikzpicture}
    \tikzset{node distance=2cm, auto}
    	\node (1) {$H_n(X_\delta,R)$};
    	\node[right of=1,xshift = 3cm] (2) {$H_n(X_{\delta'},R)$};
    	\node[below of = 1] (3) {$H_n^{\Dir}(X_\delta,R)$};
    	\node[below of = 2] (4)
    	{$H_n^{\Dir}(X_{\delta'},R)$.};
    	\draw[->] (1) to node {$H_n(i_\delta^{\delta'})$} (2);
    	\draw[->,swap] (3) to node {$H_n^{\Dir}(i_\delta^{\delta'})$} (4);
    	\draw[right hook->] (3) to node {} (1);
    	\draw[right hook->] (4) to node {} (2);
    \end{tikzpicture}
    \end{center}
    In particular, the injections
	$\big\{H_{n}^{\Dir}(X_\delta,R)\hookrightarrow H_n(X_\delta,R)\big\}_{\delta\in T}$
	give rise to a monomorphism of persistence modules.
\end{remark}

We can now introduce persistence diagrams and barcodes associated to filtrations of directed simplicial complexes. As these were only introduced for fields in Section \ref{section:persistentHomologyIntro}, from this point on, we will assume that $R$ is a field. In particular, when taking directed homology, we will assume that $R\in\{\mathbb{Q},\mathbb{R}\}$.

\begin{definition}
    Let $(X_\delta)_{\delta\in T}$ be a filtration of a directed simplicial complex $X$ where $T\subseteq\mathbb{R}$ is finite. The \emph{persistence diagrams} associated to the $n$-dimensional undirected and directed persistence $R$-modules of $X$ are respectively denoted $\Dgm_n(X,R)$ and $\Dgm_{n}^{\Dir}(X,R)$. Similarly, the respective \emph{barcodes} are denoted $\Pers_n(X,R)$ and $\Pers_{n}^{\Dir}(X,R)$.
\end{definition}

\begin{remark}
	Note that $H_0^{\Dir}(X_\delta,R)\cong H_0(X_\delta,R)$ for every $\delta\in T$, as shown in Proposition \ref{proposition:H0andConnectivity}. As a consequence, the 0-dimensional directed and undirected persistence $R$-modules of a directed simplicial complex $X$ are isomorphic. In particular, they have the same persistence barcodes and diagrams, and they measure the connectivity of the simplicial complex at each stage of the filtration.
\end{remark}

The next result establishes the relation between the undirected and directed persistence barcodes and diagrams of a persistence module, and it is thus key to understanding the directed persistence barcodes.

\begin{proposition}\label{proposition:subbarcode}
	Let $(X_\delta)_{\delta\in T}$ be a filtration of directed simplicial complexes. Then, there is an injective matching sending each bar in $\Pers_n^{\Dir}(X,R)$ to a bar in $\Pers_n(X,R)$ which contains it. Furthermore, matched bars must die at the same time.
\end{proposition}
	
\begin{proof}
	This result is an immediate consequence of Remark \ref{remark:persistenceSubmodule} and \cite[Proposition 6.1]{BauLes14}.
\end{proof}

Note that a bar in the directed persistence barcode of a filtration may be born after the one it is matched with in the undirected one, and some bars in the undirected barcode may remain unmatched (see Examples below and Figures \ref{figure:filtrationOne} and \ref{figure:notEveryDirected}). 

We now introduce some examples to illustrate Proposition \ref{proposition:subbarcode} and the general behaviour of the directed persistence barcodes. 

\begin{example}\label{example:filtrationEquivalentCycles}
    Let us illustrate how undirected and directed persistence modules and their barcodes can be different. First, consider the directed simplicial complexes $X$ and $Y$ in Figure \ref{figure:triangles} (see Example \ref{example:twoTriangles}). Regardless of the filtration chosen for $X$, the lack of directed cycles means that the $1$-dimensional directed persistence module is trivial. However, at the end of the filtration, there is a cycle in homology, thus there is a bar in the undirected persistence barcode. In the case of $Y$, the only  $1$-cycle is directed, so the undirected and directed persistence modules associated to any filtration of $Y$ are isomorphic, and different from those of $X$.
 \end{example}
 
 \begin{example}
    Consider now the directed simplicial complex $Z$ in Figure \ref{figure:triangles} (see Example \ref{example:twoTriangles}). Let $T=\{0,1,2,\dots\}$ and consider the filtration of $Z$ given by $(Z_\delta)_{\delta\in T}$, as illustrated in Figure \ref{figure:filtrationOne}, where $Z_\delta = Z$ for every $\delta \ge 3$.
        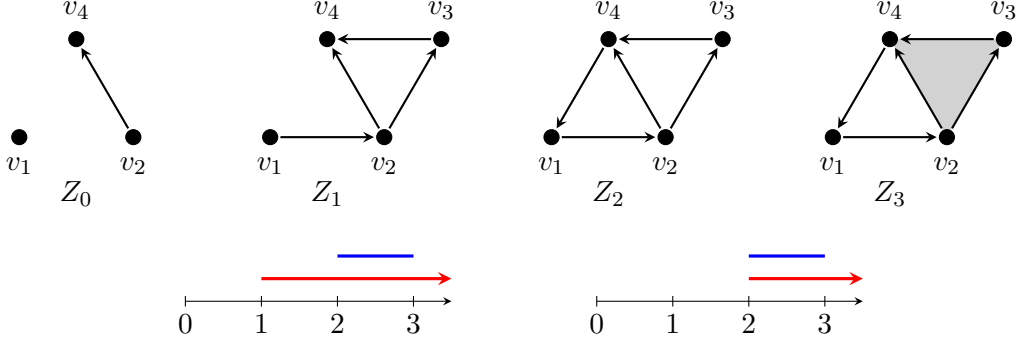
\begin{figure}
        \centering
        \begin{tikzpicture}
            \tikzset{node distance = 1.5cm, auto}
            \node[style=point,label={[label distance=0cm]270:$v_1$}](v1X1) {};
            \node[style=point, right of = v1X1, label={[label distance = 0cm]270:$v_2$}](v2X1) {};
			\node[style=point, above of = v1X1, xshift=0.75cm, yshift = -0.2 cm, label={[label distance = 0cm]90:$v_4$}](v4X1) {};
			\draw[->,thick,>=stealth] (v2X1) to node {} (v4X1);
			\node[below of = v4X1, yshift = -0.55cm] {$Z_0$};
%%%%%%%%%%%%%%%%%%%%%%%%%%%%%%%%%%%%%%%%%%%%%%%%%
			\node[style=point,right of = v2X1,xshift=0.3cm,label={[label distance=0cm]270:$v_1$}](v1X2) {};
            \node[style=point, right of = v1X2, label={[label distance = 0cm]270:$v_2$}](v2X2) {};
			\node[style=point, above of = v1X2, xshift=0.75cm, yshift = -0.2 cm, label={[label distance = 0cm]90:$v_4$}](v4X2) {};
			\node[style=point, right of = v4X2, label={[label distance = 0cm]90:$v_3$}](v3X2) {};
			\draw[->,thick,>=stealth] (v1X2) to node {} (v2X2);
           \draw[->,thick,>=stealth] (v3X2) to node {} (v4X2);
           \draw[->,thick,>=stealth] (v2X2) to node {} (v4X2);
           \draw[->,thick,>=stealth] (v2X2) to node {} (v3X2);
			\node[below of = v4X2, yshift = -0.55cm] {$Z_1$};
%%%%%%%%%%%%%%%%%%%%%%%%%%%%%%%%%%%%%%%%%%%%%%%%%
			\node[style=point,right of = v2X2,xshift = 0.7cm,label={[label distance=0cm]270:$v_1$}](v1X3) {};
            \node[style=point, right of = v1X3, label={[label distance = 0cm]270:$v_2$}](v2X3) {};
			\node[style=point, above of = v1X3, xshift=0.75cm, yshift = -0.2 cm, label={[label distance = 0cm]90:$v_4$}](v4X3) {};
			\node[style=point, right of = v4X3, label={[label distance = 0cm]90:$v_3$}](v3X3) {};
			\draw[->,thick,>=stealth] (v1X3) to node {} (v2X3);
           \draw[->,thick,>=stealth] (v4X3) to node {} (v1X3);
           \draw[->,thick,>=stealth] (v2X3) to node {} (v4X3);
           \draw[->,thick,>=stealth] (v2X3) to node {} (v3X3);
           \draw[->,thick,>=stealth] (v3X3) to node {} (v4X3);
			\node[below of = v4X3, yshift = -0.55cm] {$Z_2$};
%%%%%%%%%%%%%%%%%%%%%%%%%%%%%%%%%%%%%%%%%%%%%%%%%%
			\node[style=point,right of = v2X3,xshift = 0.7cm,label={[label distance=0cm]270:$v_1$}](v1X4) {};
            \node[style=point, right of = v1X4, label={[label distance = 0cm]270:$v_2$}](v2X4) {};
			\node[style=point, above of = v1X4, xshift=0.75cm, yshift = -0.2 cm, label={[label distance = 0cm]90:$v_4$}](v4X4) {};
			\node[style=point, right of = v4X4, label={[label distance = 0cm]90:$v_3$}](v3X4) {};
			\draw[->,thick,>=stealth] (v1X4) to node {} (v2X4);
           \draw[->,thick,>=stealth] (v4X4) to node {} (v1X4);
           \draw[->,thick,>=stealth] (v2X4) to node {} (v4X4);
           \draw[->,thick,>=stealth] (v2X4) to node {} (v3X4);
           \draw[->,thick,>=stealth] (v3X4) to node {} (v4X4);
			\node[below of = v4X4, yshift = -0.55cm] {$Z_3$};
          \begin{pgfonlayer}{background}
           \path[fill=gray!50,opacity=.7] (v3X4.center) to (v2X4.center) to (v4X4.center);
           \end{pgfonlayer}
        \end{tikzpicture}
        \bigskip\medskip
        
        \begin{tikzpicture}
            \node at (0,-0.3){0};
            \node at (1,-0.3){1};
            \node at (2,-0.3){2};
            \node at (3,-0.3){3};
            \draw[->, >=stealth] (0,0)--(3.5,0);
            \draw[-] (0,-0.07)--(0,0.07);
            \draw[-] (1,-0.07)--(1,0.07);
            \draw[-] (2,-0.07)--(2,0.07);
            \draw[-] (3,-0.07)--(3,0.07);
            \draw[->,>=stealth,very thick,color=red](1,0.3)--(3.5,0.3);
            \draw[very thick,color=blue](2,0.6)--(3,0.6);
       \end{tikzpicture}
       \qquad\qquad
       \begin{tikzpicture}
            \node at (0,-0.3){0};
            \node at (1,-0.3){1};
            \node at (2,-0.3){2};
            \node at (3,-0.3){3};
            \draw[->, >=stealth] (0,0)--(3.5,0);
            \draw[-] (0,-0.07)--(0,0.07);
            \draw[-] (1,-0.07)--(1,0.07);
            \draw[-] (2,-0.07)--(2,0.07);
            \draw[-] (3,-0.07)--(3,0.07);
            \draw[->,>=stealth,very thick,color=red](2,0.3)--(3.5,0.3);
            \draw[very thick,color=blue](2,0.6)--(3,0.6);
       \end{tikzpicture}
        \caption{A filtration of the directed simplicial complex $Z$ in Example \ref{example:twoTriangles} (Figure \ref{figure:triangles}) and its associated undirected (bottom, left) and directed (bottom, right) $1$-dimensional persistence barcodes. The shorter undirected bar (blue) is also directed, while the longer undirected bar (red) becomes directed at $\delta=2$.}
        \label{figure:filtrationOne}
    \end{figure}
    The undirected and directed $1$-dimensional persistence modules of this filtration are not isomorphic. Indeed, in $Z_1$ there is clearly an undirected cycle, whereas $Z_1^{\Dir}(Z_1,R)$ is trivial, thus $H_1^{\Dir}(Z_1,R)=\{0\}$. However, $H_1(Z_2,R)\cong H_1^{\Dir}(Z_2,R)$, as both vector spaces are generated by the classes of $[v_1,v_2]+[v_2,v_3]+[v_3,v_4]+[v_4,v_1]$ and $[v_2,v_4]+[v_4,v_1]+[v_1,v_2]$. These two classes become equivalent in $Z_3$. The undirected and directed $1$-dimensional persistence barcodes of this filtration, shown in Figure \ref{figure:filtrationOne}, illustrate how undirected classes may become directed. 
\end{example}

The last example also shows an important difference between classical and directed persistence modules and barcodes. Namely, in the undirected setting, the addition of one simplex to the filtration can either cause the birth of a class in the dimension of the added simplex, or kill a class in the preceding dimension. This simple idea is in fact the basis of the Standard Algorithm for computing persistent homology. However, the addition of only one simplex to a directed simplicial complex can cause the birth of several classes in directed homology, as shown in the previous example at $\delta=2$, and also in the following example. 

\begin{example}\label{example:oneEdgeSeveralCycles}
    Let $T=\{0,1,2,\dots\}$ and consider the filtration of simplicial complexes $(X_\delta)_{\delta\in T}$ illustrated in Figure \ref{figure:severalAtOnce} in the Introduction, where $X_\delta = X_3$ for every $\delta \ge 3$.
  
    It is clear that $Z_1^{\Dir}(X_j,R)$ is trivial for $j=0,1,2$, whereas undirected cycles appear as early as $X_1$. However, by adding the edge from $v_5$ to $v_1$, several directed cycles are born at once. Namely, $Z_1^{\Dir}(X_3,R)$ contains the cycles
    \begin{align*}
        &[v_1,v_2] + [v_2,v_3] + [v_3,v_4] + [v_4, v_5] + [v_5,v_1],\\
        &[v_1,v_2] + [v_2,v_4] + [v_4, v_5] + [v_5,v_1],\\
        &[v_1,v_2] + [v_2,v_3] + [v_3, v_5] + [v_5,v_1],\\
        &[v_1,v_3] + [v_3,v_4] + [v_4,v_5] + [v_5,v_1],\\
        &[v_1,v_3] + [v_3,v_5] + [v_5,v_1].
    \end{align*}
    Straightforward computations show that $H_1(X_3,R) = R^4$, and four of the cycles above are linearly independent, thus $H_{1}^{\Dir}(X_3,R) = H_{1}(X_3,R) = R^4$. Consequently, at $\delta = 3$ every class (including the birthing one) becomes directed.
\end{example}

Our last example shows an undirected homology class that never becomes directed.

\begin{example}\label{example:remainsUnmatched}
    Let $T=\{0,1,2,\dots\}$ and consider the filtration of simplicial complexes $(X_\delta)_{\delta\in T}$ illustrated in Figure \ref{figure:notEveryDirected}, where $X_\delta = X_2$ for $\delta\ge 2$.
    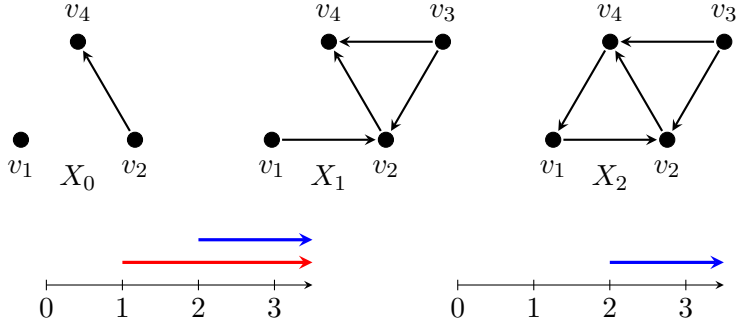
\begin{figure}
        \centering
        \begin{tikzpicture}
            \tikzset{node distance = 1.5cm, auto}
            \node[style=point,label={[label distance=0cm]270:$v_1$}](v1X1) {};
            \node[style=point, right of = v1X1, label={[label distance = 0cm]270:$v_2$}](v2X1) {};
			\node[style=point, above of = v1X1, xshift=0.75cm, yshift = -0.2 cm, label={[label distance = 0cm]90:$v_4$}](v4X1) {};
			\draw[->,thick,>=stealth] (v2X1) to node {} (v4X1);
			\node[below of = v4X1, yshift = -0.3cm] {$X_0$};
%%%%%%%%%%%%%%%%%%%%%%%%%%%%%%%%%%%%%%%%%%%%%%%%%
			\node[style=point,right of = v2X1,xshift=0.3cm,label={[label distance=0cm]270:$v_1$}](v1X2) {};
            \node[style=point, right of = v1X2, label={[label distance = 0cm]270:$v_2$}](v2X2) {};
			\node[style=point, above of = v1X2, xshift=0.75cm, yshift = -0.2 cm, label={[label distance = 0cm]90:$v_4$}](v4X2) {};
			\node[style=point, right of = v4X2, label={[label distance = 0cm]90:$v_3$}](v3X2) {};
			\draw[->,thick,>=stealth] (v1X2) to node {} (v2X2);
           \draw[->,thick,>=stealth] (v3X2) to node {} (v4X2);
           \draw[->,thick,>=stealth] (v2X2) to node {} (v4X2);
           \draw[->,thick,>=stealth] (v3X2) to node {} (v2X2);
			\node[below of = v4X2, yshift = -0.3cm] {$X_1$};
%%%%%%%%%%%%%%%%%%%%%%%%%%%%%%%%%%%%%%%%%%%%%%%%%
			\node[style=point,right of = v2X2,xshift = 0.7cm,label={[label distance=0cm]270:$v_1$}](v1X3) {};
            \node[style=point, right of = v1X3, label={[label distance = 0cm]270:$v_2$}](v2X3) {};
			\node[style=point, above of = v1X3, xshift=0.75cm, yshift = -0.2 cm, label={[label distance = 0cm]90:$v_4$}](v4X3) {};
			\node[style=point, right of = v4X3, label={[label distance = 0cm]90:$v_3$}](v3X3) {};
			\draw[->,thick,>=stealth] (v1X3) to node {} (v2X3);
           \draw[->,thick,>=stealth] (v4X3) to node {} (v1X3);
           \draw[->,thick,>=stealth] (v2X3) to node {} (v4X3);
           \draw[->,thick,>=stealth] (v3X3) to node {} (v2X3);
           \draw[->,thick,>=stealth] (v3X3) to node {} (v4X3);
			\node[below of = v4X3, yshift = -0.3cm] {$X_2$};
%%%%%%%%%%%%%%%%%%%%%%%%%%%%%%%%%%%%%%%%%%%%%%%%%%
        \end{tikzpicture}
        \bigskip\medskip
        
        \begin{tikzpicture}
            \node at (0,-0.3){0};
            \node at (1,-0.3){1};
            \node at (2,-0.3){2};
            \node at (3,-0.3){3};
            \draw[->, >=stealth] (0,0)--(3.5,0);
            \draw[-] (0,-0.07)--(0,0.07);
            \draw[-] (1,-0.07)--(1,0.07);
            \draw[-] (2,-0.07)--(2,0.07);
            \draw[-] (3,-0.07)--(3,0.07);
            \draw[->,>=stealth,very thick,color=red](1,0.3)--(3.5,0.3);
            \draw[->,>=stealth,very thick,color=blue](2,0.6)--(3.5,0.6);
       \end{tikzpicture}
       \qquad\qquad
       \begin{tikzpicture}
            \node at (0,-0.3){0};
            \node at (1,-0.3){1};
            \node at (2,-0.3){2};
            \node at (3,-0.3){3};
            \draw[->, >=stealth] (0,0)--(3.5,0);
            \draw[-] (0,-0.07)--(0,0.07);
            \draw[-] (1,-0.07)--(1,0.07);
            \draw[-] (2,-0.07)--(2,0.07);
            \draw[-] (3,-0.07)--(3,0.07);
            \draw[->,>=stealth,very thick,color=blue](2,0.3)--(3.5,0.3);
            %\draw[very thick,color=blue](2,0.6)--(3,0.6);
       \end{tikzpicture}
        \caption{A filtration of directed simplicial complexes and its associated undirected (bottom, left) and directed (bottom, right) $1$-dimensional persistence barcodes. The shorter undirected bar (blue) is also directed, while the longer undirected bar (red) never becomes directed.}
        \label{figure:notEveryDirected}
    \end{figure}
     Clearly, $Z_1^{\Dir}(X_j,R)=\{0\}$ for $j=0,1,$ whereas $Z_1^{\Dir}(X_2,R)$ contains the cycle $[v_1,v_2] + [v_2,v_4] + [v_4,v_1]$. The undirected class represented by this element is thus directed, but there is a linearly independent class in undirected homology, $[v_3,v_2] + [v_2,v_4] - [v_3,v_4]$, which never becomes directed. Its bar in the barcode is thus unmatched.
\end{example}

\subsection{Directed persistent homology of dissimilarity functions}

In this section, we introduce the undirected and directed persistence diagrams and barcodes associated to dissimilarity functions (Definition \ref{definition:persistenceNetwork}) and prove their stability with respect to the bottleneck distance (Theorem \ref{theorem:stability}). 

Let us begin by introducing the \emph{directed Rips filtration} of directed simplicial complexes associated to a dissimilarity function,  \cite[Definition 16]{Tur19}.

\begin{definition}\label{definition:filtrationDisNet}
	Let $(V,d_V)$ be a dissimilarity function. The \emph{directed Rips filtration} of $(V,d_V)$ is the filtration of directed simplicial complexes $\big(\mathcal{R}^{\Dir}(V,d_V)\big)_{\delta\in\mathbb{R}}$ where $(v_0,v_1,\dots,v_n)\in \mathcal{R}^{\Dir}(V,d_V)_\delta$ if and only if $d_V(v_i,v_j)\le \delta$, for all $0\le i \le j \le n$. It is clearly a filtration with the inclusion maps $i_\delta^{\delta'}\colon \mathcal{R}^{\Dir}(V,d_V)_\delta\to \mathcal{R}^{\Dir}(V,d_V)_{\delta'}$ for all $\delta\le\delta'$.
\end{definition}

Let us now introduce the persistent homology modules associated to such a filtration. Assume that $R\in\{\mathbb{Q},\mathbb{R}\}$.

\begin{definition}\label{definition:persistenceNetwork}
	Let $(V, d_V)$ be a dissimilarity function and consider its associated directed Rips filtration  $\big(\mathcal{R}^{\Dir}(V,d_V)\big)_{\delta\in\mathbb{R}}$. For each $n\ge 0$, the \emph{$n$-dimensional undirected persistence} $R$-module of $(V,d_V)$ is
	\[\mathcal{H}_n(V, d_V) := 	\big(\{H_n(\mathcal{R}^{\Dir}(V,d_V)_{\delta},R)\}, \{H_n(i_\delta^{\delta'})\}\big)_{\delta\le \delta'\in\mathbb{R}}.\]
	Similarly, the \emph{$n$-dimensional directed persistence} $R$-module of $(V,d_V)$ is defined as 
	\[\mathcal{H}_{n}^{\Dir}(V, d_V) := \big(\{H_{n}^{\Dir}(\mathcal{R}^{\Dir}(V,d_V)_{\delta},R)\}, \{H_{n}^{\Dir}(i_\delta^{\delta'})\}\big)_{\delta\le \delta'\in\mathbb{R}}.\]
\end{definition}

\begin{remark}\label{remark:TurnerApplicable}
    The persistence module $\mathcal{H}_n(V,d_V)$ associated to the directed Rips filtration of $(V,d_V)$ is precisely the persistence module studied in \cite[Section 5]{Tur19}, hence the remarks made there hold for the undirected persistence module. In particular, if $(V,d_V)$ is a (finite) metric space, $\mathcal{R}^{\Dir}(V,d_V)$ is the (classical) Vietoris-Rips filtration of $(V,d_V)$. Furthermore, in this case, it can easily be seen that $\mathcal{H}_n(V,d_V) = \mathcal{H}_n^{\Dir}(V,d_V)$. Thus, these persistence modules generalise the persistence modules associated to the Vietoris-Rips filtration of a metric space. Also note that $\mathcal{R}^{\Dir}(V,d_V)$ is \emph{closed under adjacent repeats}, \cite[ Definition 18]{Tur19}. Namely, given $(v_0,v_1,\dots,v_n)\in \mathcal{R}^{\operatorname{Dir}}(V, d_V)_\delta$, $(v_0,v_1,\dots,v_i,v_i,\dots,v_n)\in \mathcal{R}^{\operatorname{Dir}}(V, d_V)_\delta$, for any $i=0,1,\dots,n$. Using this, it can be proven that the homology of this object must be trivial in dimensions larger than $|V|+1$.
\end{remark}

As $R$ is a field and since $V$ is finite, both of these persistence modules fulfil the assumptions in Section \ref{section:persistentHomologyIntro}. Namely, their indexing sets can be chosen to be finite, corresponding to the threshold values where new simplices are added to the simplicial complex. Furthermore, no simplex is added to the filtration until the threshold value reaches the minimum of the images of the dissimilarity function. Finally, and even though the directed simplicial complex $\mathcal{R}^{\Dir}(V,d_V)_\delta$ may have infinite simplices due to arbitrary repetitions of vertices being allowed, it always has a finite number of simplices in a given dimension $n$, thus its $n$-dimensional homology is always finite-dimensional. As a consequence, we can introduce the following.

\begin{definition}\label{definition:directedPersBarcodes}
    Let $(V, d_V)$ be a dissimilarity function. For each $n\ge 0$, the $n$-dimensional \emph{persistence diagrams} associated to the persistence $R$-modules $\mathcal{H}_n(V, d_V)$ and $\mathcal{H}_{n}^{\Dir}(V, d_V)$ are respectively denoted by $\Dgm_n(V,d_V)$ and $\Dgm_{n}^{\Dir}(V,d_V)$. Similarly, their associated \emph{persistence barcodes} are denoted by $\Pers_n(V,d_V)$ and $\Pers_{n}^{\Dir}(V,d_V)$.
\end{definition}

Of course, Proposition \ref{proposition:subbarcode} holds for these barcodes, namely, every bar in $\Pers_n^{\Dir}(V,d_V)$ can be uniquely matched with one in $\Pers_n(V,d_V)$ which dies at the same time, although the directed bar may be born later.

We now use results from Section \ref{section:introAsymmetricNetrowks} to show that both these persistent homology constructions are stable. The proof is split in several lemmas. Let $(V,d_V)$ and $(W, d_W)$ be two dissimilarity functions on respective sets $V$ and $W$ and define $\eta = 2 d_{\CD}\big((V,d_V), (W,d_W)\big)$, where $d_{CD}$ is the correspondence distortion distance (Definition \ref{definition:networkDistance}). By Proposition \ref{proposition:networkDistance}, we can find maps $\varphi\colon V\to W$ and $\psi\colon W\to V$ such that $\dis(\varphi), \dis(\psi), \codis(\varphi,\psi),\codis(\psi,\varphi)\le \eta$. To simplify notation in the proofs below, denote $\mathcal{R}^{\Dir}(V,d_V)_\delta=X_V^\delta$ and $\mathcal{R}^{\Dir}(W,d_W)_\delta=X_W^\delta$, for all $\delta\in\mathbb{R}$.

\begin{lemma}\label{lemma:interlieved1}
	For each $\delta\in\mathbb{R}$, the maps $\varphi$ and $\psi$ induce morphisms of directed simplicial complexes
	\begin{align*}
		\begin{aligned}
			\varphi_\delta\colon X_V^\delta & \longrightarrow X_W^{\delta+\eta}\\
			x &\longmapsto \varphi(x),
		\end{aligned}&&
		\begin{aligned}
			\psi_\delta\colon X_W^\delta & \longrightarrow X_V^{\delta+\eta}\\
			x&\longmapsto \psi(x).
		\end{aligned}
	\end{align*}
\end{lemma}

\begin{proof}
	Let us prove the statement for $\varphi_\delta$ (the proof is analogous for $\psi_\delta$). Let $(x_0,x_1,\dots,x_n)$ be an $n$-simplex in $X_V^\delta$. Then $d_V(x_i,x_j)\le\delta$ for all $1\le i\le j\le n$. Since $\dis(\varphi)\le\eta$, we have that, for all $v_1,v_2\in V$,
	\[\big|d_V(v_1,v_2)-d_W\big(\varphi(v_1),\varphi(v_2)\big)\big|\le\eta.\]
	Choosing $v_1=x_i$ and $v_2=x_j$, we have 
	\[
	d_W\big(\varphi(x_i),\varphi(x_j)\big)\le \eta + d_V(x_i,x_j)\le \delta + \eta 
	\text{ for all } 1\le i \le j \le n.
	\]
	Consequently, $\big(\varphi(x_0),\varphi(x_1),\dots,\varphi(x_n)\big)\in X_W^{\delta+\eta}$ and the result follows.
\end{proof}

\begin{lemma}\label{lemma:interlieved2}
	For $\delta\le\delta'\in\mathbb{R}$ consider the inclusion maps $i_\delta^{\delta'}\colon X_V^\delta\hookrightarrow X_V^{\delta'}$ and $j_\delta^{\delta'}\colon X_W^\delta\hookrightarrow X_W^{\delta'}$. The following are commutative diagrams of morphisms of directed simplicial complexes:
	\begin{center}\begin{tikzpicture}
		\tikzset{node distance = 2.5cm, auto}
		\node(1) {$X_V^\delta$};
		\node[right of = 1,xshift=0.3cm] (2) {$X_V^{\delta'}$};
		\node[below of = 1] (3) {$X_W^{\delta + \eta}$};
		\node[below of = 2] (4) {$X_W^{\delta'+\eta}$,};
		\node[right of = 2, xshift = 1cm] (5) {$X_W^{\delta}$};
		\node[right of = 5,xshift = 0.3cm] (6) {$X_W^{\delta'}$};
		\node[below of = 5] (7) {$X_V^{\delta+\eta}$};
		\node[below of = 6] (8) {$X_V^{\delta'+\eta}$.};
		\draw[->] (1) to node{$i_\delta^{\delta'}$} (2);
		\draw[->] (3) to node{$j_{\delta+\eta}^{\delta'+\eta}$} (4);
		\draw[->] (1) to node {$\varphi_\delta$} (3);
		\draw[->] (2) to node {$\varphi_{\delta'}$} (4);
		\draw[->] (5) to node{$j_\delta^{\delta'}$} (6);
		\draw[->] (7) to node{$i_{\delta+\eta}^{\delta'+\eta}$} (8);
		\draw[->] (5) to node {$\psi_\delta$} (7);
		\draw[->] (6) to node {$\psi_{\delta'}$} (8);
	\end{tikzpicture}\end{center}
\end{lemma}

\begin{proof}
	We prove that the first diagram is commutative (the proof for the second diagram is analogous). Let $x\in V$. Since $i_\delta^{\delta'}$ is an inclusion, $(\varphi_{\delta'}\circ i_{\delta}^{\delta'})(x) = \varphi_{\delta'} (x) = \varphi(x)$. Similarly, since $j_{\delta+\eta}^{\delta'+\eta}$ is an inclusion, $(j_{\delta+\eta}^{\delta'+\eta}\circ\varphi_\delta)(x) = j_{\delta+\eta}^{\delta'+\eta}\big(\varphi(x)\big) = \varphi(x)$.
\end{proof}

\begin{lemma}\label{lemma:interlieved3}
	With the same notation as in Lemmas \ref{lemma:interlieved1} and \ref{lemma:interlieved2},  for every $\delta\in\mathbb{R}$, the following diagrams of morphisms of directed simplicial complexes induce commutative diagrams on homology.
	\begin{center}\begin{tikzpicture}
		\tikzset{node distance = 2.5cm, auto}
		\node(1) {$X_V^\delta$};
		\node[right of = 1, xshift=2cm] (3) {$X_V^{\delta+2\eta}$};
		\node[below of = 1,xshift=2.25cm] (4) {$X_W^{\delta + \eta}$,};
		\node[right of = 3, xshift = 1cm](5) {$X_W^\delta$};
		\node[right of = 5,xshift=2cm] (7) {$X_V^{\delta+2\eta}$};
		\node[below of = 5, xshift = 2.25cm] (8) {$X_W^{\delta + \eta}$.};
		\draw[->] (1) to node{$i_\delta^{\delta+2\eta}$} (3);
		\draw[->] (1) to node{$\varphi_\delta$} (4);
		\draw[->] (4) to node {$\psi_{\delta+\eta}$} (3);
		\draw[->] (5) to node{$j_\delta^{\delta+2\eta}$} (7);
		\draw[->] (5) to node{$\psi_\delta$} (8);
		\draw[->] (8) to node {$\varphi_{\delta+\eta}$} (7);
	\end{tikzpicture}\end{center}
\end{lemma}

\begin{proof}
	Again, we only prove the result for the first diagram, as the proof for the second diagram is analogous. We show that it is commutative up to homotopy by showing that the maps $i_\delta^{\delta+2\eta}$ and $\psi_{\delta+\eta}\circ\varphi_\delta$ satisfy the hypothesis of Lemma \ref{lemma:prismConstruction}. 

	Take a simplex $\sigma = (x_0,x_1,\dots,x_n)\in X_V^\delta$, thus $d_V(x_i,x_j)\le \delta$, for all $1\le i\le j\le n$. On the one hand, $i_\delta^{\delta+2\eta}$ is an inclusion, so $i_\delta^{\delta + 2\eta}(\sigma) = \sigma$. On the other hand, since $\psi_{\delta+\eta}\circ\varphi_\delta$ is a morphism of directed simplicial complexes (Definition \ref{definition:simplicialMorphism}), $\big(\psi(\varphi(x_0)),\psi(\varphi(x_1)),\dots,\psi(\varphi(x_n))\big)\in X_V^{\delta+2\eta}$. This implies that $d_V\big(\psi(\varphi(x_i)),\psi(\varphi(x_j))\big)\le \delta+2\eta$, for all $1\le i\le j\le n$.

	Now recall that $\codis(\varphi,\psi)\le\eta$, thus for all $v\in V$ and $w\in W$,
	\[\big|d_V\big(v,\psi(w)\big)-d_W\big(\varphi(v),w\big)\big|\le\eta.\]
	Then, for $1\le i \le j \le n$, by taking $v = x_i$ and $w = \varphi(x_j)$,
	\[d_V\big(x_i,\psi(\varphi(x_j))\big)\le \eta + d_W\big(\varphi(x_i),\varphi(x_j)\big)\le \delta + 2\eta.\]
	%On the other hand, we also have that $\codis(\psi,\varphi)\le\eta$, so that, for every $v\in V$ and $w\in W$,
	%\[\big|d_W\big(w,\varphi(v)\big)-d_V\big(\psi(w),v\big)\big|\le\eta.\]
	%Consequently, for $1\le i\le j\le n$, by taking $v = x_j$ and $w = \varphi(x_i)$,
	%\[d_V\big(\psi(\varphi(x_i)),x_j\big)\le d_W\big(\varphi(x_i),\varphi(x_j)\big)+\eta \le \delta+2\eta.\]
	As a consequence of the inequalities above, we have shown that for every $0\le i \le n$, \[\big(x_0,x_1,\dots,x_i,\psi(\varphi(x_i)),\psi(\varphi(x_{i+1})),\dots,\psi(\varphi(x_n))\big)\in X_V^{\delta+2\eta}.\]

	Therefore, the maps $i_\delta^{\delta+2\eta}$ and $\psi_{\delta+\eta}\circ\varphi_\delta$ satisfy the hypothesis of Lemma \ref{lemma:prismConstruction}, thus they induce the same map on homology. The result follows.
\end{proof}

We now have everything we need to prove the stability results.

\begin{theorem}\label{theorem:stability}
	Let $R\in\{\mathbb{Q},\mathbb{R}\}$. Let $(V, d_V)$ and $(W, d_W)$ be two dissimilarity functions on finite sets $V$ and $W$. Then, for all $n\ge 0$,
	\[d_B\big(\Dgm_n(V,d_V),\Dgm_n(W,d_V)\big)\le 2 d_{\CD}\big((V,d_V),(W,d_W)\big)\]
	and
	\[d_B\big(\Dgm_{n}^{\Dir}(V,d_V),\Dgm_{n}^{\Dir}(W,d_V)\big)\le 2 d_{\CD}\big((V,d_V),(W,d_W)\big).\]
\end{theorem}

\begin{proof}
	Define $\eta = 2 d_{\CD}\big((V,d_V),(W,d_W)\big)$. By Theorem \ref{theorem:interlievedImpliesStability}, it suffices to show that the persistence modules $\mathcal{H}_n(V, d_V)$ and $\mathcal{H}_n(W, d_W)$ (respectively $\mathcal{H}_{n}^{\Dir}(V, d_V)$ and $\mathcal{H}_{n}^{\Dir}(W, d_W)$) are $\eta$-interleaved. Comparing Definition \ref{definition:bottleneckDistance} (for $\varepsilon=\eta$) and Lemmas  \ref{lemma:interlieved1}, \ref{lemma:interlieved2} and \ref{lemma:interlieved3}, the result follows by using the functoriality of homology and, in the directed case, Proposition \ref{proposition:degmorphism}.
\end{proof}

\begin{remark}\label{remark:stabilityAndRepetitions}
   Recall from Remark \ref{remark:TurnerApplicable} that the persistence module $\mathcal{H}_n(V,d_V)$ associated to the directed Rips filtration of $(V,d_V)$ is the persistence module studied in \cite[Section 5]{Tur19}. Thus, the remarks made in \cite[Section 5.2]{Tur19} hold for these persistence modules, meaning that a result analogous to Theorem \ref{theorem:stability} would not hold if we were using ordered-set complexes instead of directed simplicial complexes, as mentioned at the beginning of Section \ref{section:homologySimplicialComplexes}. This justifies our definition of directed simplicial complex (Definition \ref{definition:directedSimplicialComplex}).
\end{remark}

\section{Computational challenges}\label{section:algorithms}

Let $(V, d_V)$ be a dissimilarity function in a finite set $V$. In this section, we study the algorithmic aspects of computing both $\mathcal{H}_{n}(V, d_V)$ and $\mathcal{H}_{n}^{\Dir}(V, d_V)$. For simplicity, we assume that $R =\mathbb{R}$.

The computation of the persistence barcodes for the undirected persistent homology $\mathcal{H}_{n}(V, d_V)$ can be made using the Standard Algorithm \cite{EdeLetZom02,ZomCar05}, which also applies to ordered simplicial complexes. 
Write $\mathcal{R}^{\Dir}(V,d_V)_\delta = X_V^\delta$, $\delta\in\mathbb{R}$ for the directed Rips filtration (Definition \ref{definition:filtrationDisNet}). As $V$ is finite, the number of simplices in $X_V^\delta$ up to dimension $n+1$ is finite, say $N$. We can list them $\{\sigma_1,\sigma_2,\dots,\sigma_N\}$ in such a way that $i < j$ if $\sigma_i$ is a (proper) face of $\sigma_j$, or if $\sigma_j$ appears `later' in the filtration, thus having a \emph{compatible ordering}. Then, we can represent the differential using a sparse $N\times N$ matrix $M$ over $\mathbb{R}$, where the $(i,j)$-entry $M_{ij}$ is the coefficient of $\sigma_j$ in the differential of $\sigma_i$. Such matrix $M$ is an upper-triangular matrix. At this stage, the Standard Algorithm can be used to compute $\Dgm_k(V,d_V)$ by reducing $M$ at once using column operations.

We cannot directly adapt the Standard Algorithm to compute $\Dgm_k^{\Dir}(V,d_V)$: we can try to find reduction operations that generate directed cycles using operations with positive coefficients only, but this is not enough to guarantee that we find all the possible directed cycles. For instance, once one column has been eliminated, we may be preventing the participation of the simplex it represents in directed cycles not involving the columns we are using in the elimination. To avoid such a problem, we would need to keep track of every possible way in which column eliminations can be produced using only positive coefficients. This is very similar to the \emph{double description algorithm} for the computation of extreme rays of a polyhedral cone \cite{FukPro96,MotRaiThoThr53}. Indeed, 
the directed cycles $Z_n^{\Dir}(X,\mathbb{R})$ are the non-negative solutions to the linear equation $Mx = 0$, that is,
\begin{equation}\label{eq:positiveCyles}
    \big\{x = (x_1,x_2,\dots,x_N)\in \mathbb{R}^N \mid M x = 0, x\ge 0\big\},
\end{equation}
where a vector $x = (x_1,x_2,\dots,x_N) \in \mathbb{R}^N$ represents the cycle $\sum_{i=1}^N x_i \sigma_i$, $M$ is the matrix corresponding to the differential, and
$x \ge 0$ means that every entry in $x$ is non-negative, that is, $x_i \ge 0$ for all $i$. Equivalently, directed cycles are the extreme rays of the unbounded polyhedral cone consisting of the intersection of the solution space to the linear system $M x = 0$ with the first quadrant of $\mathbb{R}^N$.
The enumeration of the extreme rays (generators of the solution set) of a polyhedral cone such as Eq.~\eqref{eq:positiveCyles} is a well-known albeit hard problem in computational geometry. Several algorithms, including the double description algorithm, and modifications thereof, exist \cite{Ter09}. 

Note, however, that we do not require the computation of all the directed cycles. Indeed, it would be enough to calculate them up to homology, which suggests that efficient computations may be possible. %Another possibility is using semiring theory, as the ideas for directed homology in this paper have their origin in homology with semiring coefficients, as explained in Appendix \ref{appendix:algebraicBackground}.

In summary, although important challenges remain on the computation of the persistence diagrams associated to directed persistence modules, our objective was to provide the necessary groundwork for the study of directed persistent homology of asymmetric data sets, and we are hopeful that we are opening up an exciting new approach for future research into the topological properties of asymmetric data sets. 

\subsection*{Acknowledgements}
This work was supported by The Alan Turing Institute under the EPSRC grant EP/N510129/1. The first author was partially supported by Ministerio de Econom\'{i}a y Competitividad (Spain) Grants Nos.\ MTM2016-79661-P and MTM2016-78647-P and by Ministerio de Ciencia, Innovaci\'on y Universidades (Spain) Grant No.\ PGC2018-095448-B-I00.

\appendix
\section{Directed persistent homology via semiring homology}\label{appendix:algebraicBackground}

As mentioned in the Introduction, our ideas of directed homology are based on homology with semiring coefficients, which can be introduced for directed simplicial complexes using chain complexes of semimodules and their associated homologies, building on work by Patchkoria \cite{Pat77,Pat14}. In this appendix, we introduce this homology theory and explain how it relates to the modules of directed homology introduced in Definition \ref{definition:directed-homology}. 

\subsection{Semirings, semimodules and their completions}\label{appendixSection:semimodules} Let us begin by introducing the algebraic structures that we need, following \cite{Gol99}.

\begin{definition}\label{definition:semiring}
	A \emph{semiring} $\Lambda = (\Lambda, +, \cdot)$ is a set $\Lambda$ together with two operations such that
	\begin{itemize}
		\item $(\Lambda,+)$ is an abelian monoid whose identity element we denote $0_\Lambda$,
		\item $(\Lambda,\cdot)$ is a monoid whose identity element we denote $1_\Lambda$,
		\item $\cdot$ is distributive with respect to $+$ from either side,
		\item $0_\Lambda\cdot \lambda = \lambda \cdot 0_\Lambda = 0_\Lambda$, for all $\lambda\in\Lambda$.
	\end{itemize}

	A semiring $\Lambda$ is \emph{commutative} if $(\Lambda,\cdot)$ is a commutative monoid, and \emph{cancellative} if $(\Lambda,+)$ is a cancellative monoid, that is, $\lambda + \lambda' = \lambda + \lambda''$ implies $\lambda' = \lambda''$ for all $\lambda,\lambda',\lambda''\in\Lambda$ . A semiring $\Lambda$ is a \emph{semifield} if every $0_\Lambda \neq \lambda\in\Lambda$ has a multiplicative inverse. A semiring is \emph{zerosumfree} if no element other than $0_\Lambda$ has an additive inverse.
\end{definition}

\begin{example}
	Every ring is a semiring. The non-negative integers, rationals and reals with their usual operations, respectively denoted $\mathbb{N}$, $\mathbb{Q}^+$ and $\mathbb{R}^+$, are cancellative zerosumfree commutative semirings which are not rings.
\end{example}

\begin{definition}\label{definition:semimodule}
	Let $\Lambda$ be a semiring. A \emph{(left) $\Lambda$-semimodule} is an abelian monoid $(A,+)$ with identity element $0_A$ together with a map $\Lambda\times A\to A$ which we denote $(\lambda,a)\mapsto \lambda a$ and such that for all $\lambda,\lambda'\in\Lambda$ and $a,a'\in A$,
	\begin{itemize}
		\item $(\lambda\lambda')a = \lambda(\lambda'a)$,
		\item $\lambda(a+a') = \lambda a + \lambda a'$,
		\item $(\lambda + \lambda')a = \lambda a + \lambda' a$,
		\item $1_\Lambda a = a$,
		\item $\lambda 0_A = 0_A = 0_\Lambda\lambda$.
	\end{itemize}

	A non-empty subset $B$ of a left $\Lambda$-semimodule $A$ is a \emph{subsemimodule} of $A$ if $B$ is closed under addition and scalar multiplication, which implies that $B$ is a left $\Lambda$-semimodule with identity element $0_A\in B$.	If $A$ and $B$ are $\Lambda$-semimodules, a $\Lambda$-homomorphism is a map $f\colon A\to B$ such that for all $a,a'\in A$ and for all $\lambda\in\Lambda$,
	\begin{itemize}
		\item $f(a+a') = f(a) + f(a')$,
		\item $f(\lambda a) = \lambda f(a)$.
	\end{itemize}
	Clearly, $f(0_A) = 0_B$.

	A $\Lambda$-semimodule $A$ is \emph{cancellative} if $a+a'=a+a''$ implies $a'=a''$, and \emph{zerosumfree} if $a+a'=0$ implies $a = a' = 0$, for all $a,a',a''\in A$.
\end{definition}

The Cartesian product of $\Lambda$-semimodules is a $\Lambda$-semimodule. The definition of direct sum of modules can be generalised to the framework of $\Lambda$-semimodules, and a finite direct sum of $\Lambda$-semimodules is isomorphic to the Cartesian product of the same $\Lambda$-semimodules. Quotient $\Lambda$-semimodules can be defined using  congruence relations.

\begin{definition}\label{definition:congruence}
	Let $A$ be a left $\Lambda$-semimodule. An equivalence relation $\rho$ on $A$ is a \emph{$\Lambda$-congruence} if, for all $a, a' \in \Lambda$, and all $\lambda \in \Lambda$,   
	\begin{itemize}
		\item if $a \sim_\rho a'$ and $b \sim_\rho b'$, then $(a + b)\sim_\rho (a' + b')$, and
		\item if $a \sim_\rho a'$, then $\lambda a \sim_\rho \lambda a'$.
	\end{itemize}
	If $\rho$ is a $\Lambda$-congruence relation on $A$ and we write $a/\rho$ for the class of an element $a\in A$, then $A/\rho =\{a/\rho\mid a\in A\}$ inherits a $\Lambda$-semimodule structure by setting $(a/\rho)+(a'/\rho) = (a+a')/\rho$ and $\lambda(a/\rho) = (\lambda a)/\rho$, for all $a,a'\in A$ and $\lambda\in\Lambda$. The left $\Lambda$-semimodule $A/\rho$ is called the \emph{factor semimodule} of $A$ by $\rho$. Note that the quotient map $A\to A/\rho$ is a surjective $\Lambda$-homomorphism.

	If $B$ is a subsemimodule of a $\Lambda$-semimodule $A$, then it determines a $\Lambda$-congruence $\sim_B$ by setting $a \sim_B a'$ if there exist $b,b'\in B$ such that $a + b = a' + b'$. Classes in this quotient are denoted $a/B$, and the factor semimodule is denoted $A/B$. If $\Lambda$ is a ring, this is just the usual quotient module of $A$ by $B$.
\end{definition}

\begin{example}\label{examples:factorSemimodules} The following are examples of factor semimodules.
    \begin{itemize}
        \item If $\Lambda$ is a ring and $B\le A$ are $\Lambda$-modules, the factor semimodule $A/B$ is the usual module quotient of $A$ by $B$.
        
        \item In the $\mathbb{N}$-module $\mathbb{N}\times\mathbb{N}$, consider the relation $\rho$ such that $(u,v)\sim_\rho(x,y)$ if $u + y = v + x$. It is immediate to show that $\rho$ is a $\mathbb{N}$-congruence, and the factor semimodule $(\mathbb{N}\times\mathbb{N})/\rho$ is isomorphic to $\mathbb{Z}$. Indeed, consider the map $(\mathbb{N}\times\mathbb{N})/\rho\to\mathbb{Z}$ taking $(u,v)\mapsto u - v$. It is clearly a map of $\mathbb{N}$-semimodules, $(u,v)\sim_\rho(x,y)\Leftrightarrow u + y = v + x\Leftrightarrow u - v = x - y$ shows that it is well-defined and injective, and it is clearly surjective.
        
        \item Also in $\mathbb{N}\times\mathbb{N}$, consider the submodule $B = \{0\}\times\mathbb{N}$. Then $(u,v)\sim_B(x,y)$ if there exist $(0,n),(0,m)\in \{0\}\times\mathbb{N}$ such that $(u,v) + (0,n) = (x,y) + (0,m)$. From this, we deduce that $(u,v)\sim_B(x,y)$ if and only if $u = x$. Namely, $(\mathbb{N}\times\mathbb{N})/(\{0\}\times\mathbb{N})$ is isomorphic to $\mathbb{N}$.
    \end{itemize}
\end{example}

Clearly, if $\Lambda$ is a semiring then it is a $\Lambda$-semimodule over itself. The idea of free $\Lambda$-semimodules comes about naturally just as in the case of rings, and we will make extensive use of these objects.

%Let $A$ be a $\Lambda$-semimodule and $X$ be a set. Let $A^{(X)}$ denote the set of functions from $X$ to $A$ with finite support, namely the elements of $A^{(X)}$ are maps from $X$ to $A$ such that $f(x) = 0$ for all but a finite amount of elements $x\in X$. Then $A^{(X)}$ is a $\Lambda$-semimodule with operations $(f + g)(x) = f(x) + g(x)$ and $(\lambda f)(x) = \lambda f(x)$, for all $f,g\in A^{(X)}$, $x\in X$ and $\lambda\in\Lambda$.

\begin{definition}\label{definition:freesemimodule}
	Let $\Lambda$ be a semiring, $A$ a left $\Lambda$-semimodule and $V = \{v_1,v_2,\dots,v_n\}$ a finite subset of $A$. The set $V$ is a \emph{generating set} for $A$ if every element in $A$ is a linear combination of elements of $V$. The \emph{rank} of a $\Lambda$-semimodule $A$, denoted $\rank(A)$, is the least $n$ for which there is a set of generators of $A$ with cardinality $n$, or infinity, if not such $n$ exists.
	
	The set $V$ is \emph{linearly independent} if for any $\lambda_1,\lambda_2,\dots,\lambda_n\in\Lambda$ and $\mu_1,\mu_2,\dots,\mu_n\in\Lambda$ such that $\sum_{i=1}^n \lambda_i v_i = \sum_{i=1}^n \mu_i v_i$, then $\lambda_i=\mu_i$, for $i=1,2,\dots,n$, and it is called \emph{linearly dependent} otherwise. We call $V$ is a \emph{basis} of $A$ if it is a linearly independent generating set of $A$. 

	The $\Lambda$-semimodule $A$ is a \emph{free $\Lambda$-semimodule} if it admits a basis $V$. We denote a free $\Lambda$-semimodule with basis $\{v_1,v_2,\dots,v_n\}$ by $\Lambda(v_1,v_2,\dots,v_n)$.
\end{definition}

\begin{remark}
	It is easy to check that free $\Lambda$-semimodules over a cancellative semiring $\Lambda$ are themselves cancellative, as are subsemimodules of cancellative $\Lambda$-semimodules. The quotient of a cancellative $\Lambda$-semimodule over any of its subsemimodules is also cancellative, see \cite[Proposition 15.24]{Gol99}.
\end{remark}

Note that if $\Lambda$ is a ring, free $\Lambda$-semimodules are just free $\Lambda$-modules. Thus, as not every module over a ring is free, clearly not every $\Lambda$-semimodule is free. On the other hand, $\Lambda^n$ (the direct sum, or product, of $n$ copies of $\Lambda$) is clearly a free $\Lambda$-semimodule, for every $n\ge 1$. 
Another key property is that if $A$ is a free $\Lambda$-semimodule with basis $V$ and $B$ is another $\Lambda$-semimodule, each map $V\to B$ uniquely extends to a $\Lambda$-homomorphism $A\to B$, see \cite[Proposition 17.12]{Gol99}. 

\begin{remark}\label{remark:basis}
	If $\Lambda$ is a commutative semiring and $A$ is a $\Lambda$-semimodule admitting a finite basis, every basis has the same cardinality, which coincides with the rank of $A$, \cite[Theorem 3.4]{Tan14}. In fact, for every integer $n>0$, every finitely generated subsemimodule of $\mathbb{N}^n$ has a unique basis, whereas basis for semimodules of $(\mathbb{R}^+)^n$ are unique up to non-zero multiples, see e.g.~\cite[Theorem 2.1]{KimRou80}.
\end{remark}

We finish this section by extending the Grothendieck construction from abelian monoids to semirings and semimodules. 

\begin{definition}\label{definition:grothendieckCompletion}
	Let $M$ be an abelian monoid. Consider the congruence relation $\rho$ in $M\times M$ defined as
	\[(u,v)\sim_\rho (x,y)\Leftrightarrow \text{there exists $z\in M$ such that  $u + y + z = v + x + z$}.\]
	Let $[u,v]$ denote the equivalence class of $(u,v)$. Then $M\times M/\rho$ becomes a group with the componentwise addition. This group is called the \emph{Grothendieck group} or \emph{group completion} of $M$. There is a canonical homomorphism of monoids $k_M\colon M\to K(M)$ defined as $k_M(x) = [x,0]$. If $M$ is a cancellative monoid, then $k_M$ is injective.

	Given an element $[u,v]\in K(M)$ we can interpret $u$ as its \emph{positive part} and $v$ as its \emph{negative part}, and the relation $\sim$ becomes the obvious one. The identity element is then $0 = [x,x]$ for any element $x\in M$, and the inverse of $[x,y]$ is $[y,x]$ for any $x,y\in M$.

	If $f\colon M\to N$ is a morphism of monoids, there is a morphism of groups $K(f)\colon K(M)\to K(N)$ that takes $[u,v]$ to $[f(u),f(v)]$. In fact, $K$ is a functor from the category of abelian monoids to the category of abelian groups. Furthermore, $K(f)\circ k_M = k_N\circ f$.
\end{definition}

As shown in Example \ref{examples:factorSemimodules},  $K(\mathbb{N})\cong\mathbb{Z}$. Similarly, one can show that $K(\mathbb{R}^+)\cong\mathbb{R}$. Crucially for us, this construction can be extended to semirings and semimodules.

\begin{definition}\label{definition:semimodule-completion}
	Let $\Lambda$ be a semiring. The group completion $K(\Lambda)$ becomes a ring with the operation $[x_1,x_2]\cdot[y_1,y_2] = [x_1 y_1 + x_2 y_2, x_2 y_1 + x_1 y_2]$, and $k_\Lambda\colon\Lambda\to K(\Lambda)$ is in fact a morphism of semirings, which is injective if $\Lambda$ is cancellative. We call $K(\Lambda)$ the \emph{ring completion} of $\Lambda$. If $f\colon\Lambda\to\Lambda'$ is a morphism of semirings, the map $K(f)\colon K(\Lambda)\to K(\Lambda')$ as introduced in Definition \ref{definition:grothendieckCompletion} is a morphism of rings, and $K$ is a functor from the category of semirings to the category of rings.

	If $A$ is a $\Lambda$-semimodule, then the abelian group $K(A)$ together with the operation $K(\Lambda)\times K(A)\to K(A)$ given by $[\lambda_1,\lambda_2][a_1,a_2] = [\lambda_1 a_1 + \lambda_2 a_2, \lambda_1 a_2 + \lambda_2 a_1]$ is a $K(\Lambda)$-module, the $K(\Lambda)$-module completion of $A$. Again, if $f\colon A\to B$ is a $\Lambda$-morphism, the map $K(f)\colon K(A)\to K(B)$ becomes a morphism of $K(\Lambda)$-modules. Furthermore, if $A$ is a free $\Lambda$-semimodule with basis $\{v_i\mid i\in I\}$, it becomes immediate that $K(A)$ is a free $K(\Lambda)$-module with basis $\big\{[v_i,0]\mid i\in I\}$. 
\end{definition}

%One final consideration is that if $\Lambda$ is a cancellative semifield, then $K(\Lambda)$ is a field.

\subsection{Chain complexes of semimodules}\label{appendixSection:chainComplexes}

Now that we have introduced the necessary algebraic structures, we move on to chain complexes of semimodules over semirings, following \cite{Pat14}. This theory is a natural generalisation of the classical theory of chain complexes of modules and, in fact, they give rise to the same cycles, boundaries and homologies when the semimodules are modules over a ring. 

In order to introduce chain complexes in the context of semimodules an immediate problem arises: alternating sums cannot be defined as elements in a semimodule may not have inverses. The solution is to use two maps, a \emph{positive} and \emph{negative} part, for the differentials.

\begin{definition}
	Let $\Lambda$ be a semiring and consider a sequence of $\Lambda$-semimodules and homomorphisms indexed by $n \in \mathbb{Z}$
	\[X\colon\cdots \rightrightarrows X_{n+1} \xrightrightarrows[\partial_{n+1}^-]{\partial_{n+1}^+} X_n \xrightrightarrows[\partial_{n}^-]{\partial_{n}^+} X_{n-1} \rightrightarrows \cdots. \]
	We say that $X = \{X_n, \partial_n^+, \partial_n^-\}$ is a \emph{chain complex} of $\Lambda$-semimodules if
	\[\partial^+_n \partial^+_{n+1} + \partial_n^-\partial_{n+1}^- = \partial_n^+ \partial_{n+1}^- + \partial_n^-\partial_{n+1}^+.\]
\end{definition}

As in the classical case, chain complexes of $\Lambda$-semimodules give rise to a $\Lambda$-semimodule of homology.

\begin{definition}\label{definition:homology}
	Let $X = \{X_n, \partial_n^+, \partial_n^-\}$ be a chain complex of $\Lambda$-semimodules. The $\Lambda$-semimodule of \emph{cycles} of $X$ is
	\[Z_n(X,\Lambda) = \{x\in X_n\mid \partial_n^+(x) = \partial_n^-(x)\}.\]
	The \emph{$n$th homology} of $X$ is then the quotient $\Lambda$-semimodule
	\[H_n(X,\Lambda) = Z_n(X,\Lambda)/\rho_n(X,\Lambda)\]
	where $\rho_n(X,\Lambda)$ is the following $\Lambda$-congruence relation on $Z_n(X,\Lambda)$:
	\[x \sim_{\rho_n(X,\Lambda)} y\Leftrightarrow \exists u,v\in X_{n+1} \text{ s.t. }  x + \partial_{n+1}^+(u) + \partial_{n+1}^-(v) = y + \partial_{n+1}^+(v) + \partial_{n+1}^-(u).\]
	We will omit from now on the coefficient semiring $\Lambda$ from the notation when it is clear from the context.
\end{definition}

\begin{remark}
  The definition of cycle is a direct generalisation of the classical definition. For the boundary relation, note that we may need two different chains $u$ and $v$ in order to establish two classes as homologous. Intuitively, these two classes are the `positive' and `negative' part of $w$, where $x = y + \partial(w)$ in the classical setting.
\end{remark}

% \ruben{I would still like to understand the relation with the simpler congruence
% \[ 
%     x \sim y \iff \exists u \text{ s.t. }
%     x+\partial^-(u) = y + \partial^+(u)
% \]
% but happy to leave it for another time. 
% }

\begin{remark}\label{remark:homologysemiringsandrings}
	If $X = \{X_n, \partial_n^+, \partial_n^-\}$ is a chain complex of $\Lambda$-semimodules, then
	\[\dots\longrightarrow K(X_{n+1})\xrightarrow{K(\partial_{n+1}^+) - K(\partial_{n+1}^-)} K(X_n)\xrightarrow{K(\partial_{n}^+) - K(\partial_{n}^-)} K(X_{n-1})\longrightarrow \cdots\]
	is a chain complex of $K(\Lambda)$-modules. If furthermore the $\Lambda$-semimodules $X_n$ are cancellative for all $n$, the converse is also true.

	If $\Lambda$ is a ring, then $K(\Lambda) = \Lambda$ and the functor $K$ acts as the identity on $\Lambda$-modules. Therefore, $X = \{X_n, \partial_n^+, \partial_n^-\}$ is a chain complex of $\Lambda$-semimodules if and only if
	\[\cdots\longrightarrow X_{n+1}\xrightarrow{\partial_{n+1}^+ - \partial_{n+1}^-} X_n \xrightarrow{\partial_{n}^+ - \partial_{n}^-} X_{n-1}\longrightarrow \cdots\]
	is a chain complex of $\Lambda$-modules. In this case, it is clear that the homology semimodules introduced in Definition \ref{definition:homology} are precisely the usual homology modules of $X$.
\end{remark}

We remark that a similar idea of decomposition of differentials in a positive and negative components has been used by Steiner to relate $\omega$-categories with the so-called \emph{augmented directed complexes}, \cite{Ste04}. These structures share, in fact, a deep similarity with Patchkoria's chain complexes with semiring coefficients. 

We now turn our attention to maps between complexes. 

\begin{definition}
	Let $X = \{X_n, \partial_n^+, \partial_n^-\}$ and $X' = \{X'_n, \partial_n^+, \partial_n^-\}$ be two chain complexes of $\Lambda$-semimodules. A sequence $f = \{f_n\}$ of $\Lambda$-homomorphisms $f_n\colon X_n\to X_n'$ is said to be a \emph{morphism} from $X$ to $X'$ if 
	\[\partial_n^+ f_n + f_{n-1}\partial_n^- = \partial_n^- f_n + f_{n-1}\partial_n^+.\]
	
	It is clear that for such a map  $f_n\big(Z_n(X)\big)$ is a $\Lambda$-subsemimodule of $Z_n(X')$. Furthermore, if $X_n$ and $X_n'$ are cancellative $\Lambda$-semimodules, $f$ is also compatible with the congruence relations $\rho_n(X)$ and $\rho_n(X')$, so it induces a map
	\[H_n(f)\colon H_n(X)\to H_n(X').\]
	It is then easy to check that homology $H_*$ is a functor from the category of chain complexes of cancellative $\Lambda$-semimodules and morphisms to the category of graded $\Lambda$-semimodules.
\end{definition}

\begin{remark}\label{remark:homologyandKfunctor}
	Let $\Lambda$ be a semiring and let $\{X_n, \partial_n^+, \partial_n^-\}$ be a chain complex of $\Lambda$-semimodules. The family of canonical maps $k_{X_n}\colon X_n\to K(X_n)$ gives rise to a morphism from $\{X_n, \partial_n^+, \partial_n^-\}$ to $\big\{K(X_n), K(\partial_n^+), K(\partial_n^-)\big\}$ and, therefore, to a morphism of $\Lambda$-semimodules
	\[H(k_X)\colon H_n(X)\to H_n\big(K(X)\big),\]
	which takes the class of $x$ to the class of $[x,0]$. Furthermore, if the $\Lambda$-semimodules $X_n$ are cancellative, $H_n(k_X)$ is injective, which in particular implies that $H_n(X)$ is a cancellative $\Lambda$-semimodule. Also note that, by Remark \ref{remark:homologysemiringsandrings}, the homology of $\big\{K(X_n), K(\partial_n^+), K(\partial_n^-)\big\}$ and the usual homology of $\big\{K(X_n), K(\partial_n^+)-K(\partial_n^-)\big\}$ are isomorphic as $K(\Lambda)$-modules.
\end{remark}

Finally, we discuss chain homotopies.

\begin{definition}\label{definition:chainHomotopy}
	Let $f = \{f_n\}$ and $g = \{g_n\}$ be morphisms from $X = \{X_n, \partial_n^+, \partial_n^-\}$ to $X' = \{X'_n, \partial_n^+, \partial_n^-\}$.  We say that $f$ is \emph{homotopic} to $g$ if there exist $\Lambda$-homomorphisms $s_n^+, s_n^-\colon X_n\to X_{n+1}'$ such that
	\begin{align*}
		& \partial_{n+1}^+ s_n^- + \partial_{n+1}^- s_n^+ + s_{n-1}^- \partial_n^++ s_{n-1}^+\partial_n^- + g_n\\
		= \ & \partial_{n+1}^+ s_n^+ + \partial_{n+1}^- s_n^- + s_{n-1}^+ \partial_n^++ s_{n-1}^-\partial_n^- + f_n, 
	\end{align*}
	for all $n$. The family $\{s_n^+,s_n^-\}$ is called a \emph{chain homotopy} from $f$ to $g$, and we write $(s^+,s^-)\colon f\simeq g$.
\end{definition}

We then have the following result.

\begin{proposition}{\cite[Proposition 3.3]{Pat14}}	Let $f,g\colon X\to X'$ be morphisms between chain complexes of cancellative $\Lambda$-semimodules. If $f$ is homotopic to $g$, then $H_n(f) = H_n(g)$.
\end{proposition}

Homotopy equivalences are defined in the usual way and they induce isomorphisms on homology. We finish with a remark that the homotopy of maps behaves well with respect to semimodule completion.

\begin{remark}
	If a morphism $f\colon X\to X'$ is homotopic to $g\colon X\to X'$, then $K(f)\colon K(X)\to K(X')$ is homotopic to $K(g)\colon K(X)\to K(X')$. Furthermore, if both $X_n$ and $X_n'$ are cancellative $\Lambda$-semimodules, for all $n$, and $X_n$ is a free $\Lambda$-semimodule for all $n$, the converse is also true.
\end{remark}

\subsection{Chain complexes of semimodules of directed simplicial complexes}

In this section, we introduce chain complexes of semimodules and homology semimodules of directed simplicial complexes, Definition \ref{definition:directedSimplicialComplex}. The results introduced here directly generalise those in Section \ref{section:homologySimplicialComplexes}.

\begin{definition}
	The \emph{$n$-dimensional chains} of a directed simplicial complex $X$ are defined as the elements of the free $\Lambda$-semimodule generated by (i.e.~with basis) the $n$-simplices,
	\[C_n(X,\Lambda) = \Lambda\big(\left\{[x_0,x_1,\dots,x_n]\mid
	(x_0,x_1,\dots,x_n)\in X\right\}\big).\]
	We call the elements $[x_0,x_1,\dots,x_n]\in C_n(X,\Lambda)$, $(x_0,x_1,\dots,x_n)\in X$, \emph{elementary $n$-chains}.
\end{definition}

\begin{remark}\label{remark:cancellative}
    Note that if $\Lambda$ is a cancellative semiring, the $\Lambda$-semimodule $C_n(X,\Lambda)$ is cancellative for every $n$, as it is a free semimodule over a cancellative semiring.
\end{remark}

We now define the positive and negative differentials on $C_n(X,\Lambda)$.

\begin{definition}
	Let $X$ be a directed simplicial complex. For each $n>0$, we define morphisms of $\Lambda$-semimodules $\partial_n^+,\partial_n^-\colon C_n(X,\Lambda)\to C_{n-1}(X,\Lambda)$ by
	\begin{align*}
	\partial_n^+([x_0,x_1,\dots,x_n]) &= \sum_{i=0}^{\lfloor\frac{n}{2}\rfloor}[x_0,x_1,\dots,\widehat{x_{2i}},\dots,x_n], \text{ and}\\
	\partial_n^-([x_0,x_1,\dots,x_n]) &= \sum_{i=0}^{\lfloor\frac{n-1}{2}\rfloor}[x_0, x_1,\dots,\widehat{x_{2i+1}},\dots,x_n].
	\end{align*}
    For $n=0$, let $\partial_0^+,\partial_0^-\colon C_0(X,\Lambda)\to\{0\}$ be the trivial maps, by definition.
\end{definition}

\begin{proposition}\label{proposition:semiringchaincomplex}
	Let $X$ be a directed simplicial complex and $\Lambda$ be a cancellative semiring. Then $\{C_n(X,\Lambda),\partial_n^+,\partial_n^-\}$ is a chain complex of modules of $\Lambda$-semimodules.			
\end{proposition}

\begin{proof}
    By Remark \ref{remark:cancellative}, the $\Lambda$-semimodule $C_n(X,\Lambda)$ is cancellative for every $n$. Thus, by Remark \ref{remark:homologysemiringsandrings}, it is enough to prove that $\big\{K\big(C_n(X,\Lambda)\big),K(\partial_n^+)-K(\partial_n^-)\big\}$ is a chain complex of $K(\Lambda)$-modules.
    Note that $K\big(C_n(X,\Lambda)\big)$ is a free $K(\Lambda)$-module whose basis is given by elements $\big[[x_0,x_1,\dots,x_n],0\big]$ such that $[x_0,x_1,\dots,x_n]$ is an elementary $n$-chain. Thus, it suffices to show that the composition $\big(K(\partial_n^+)-K(\partial_n^-)\big) \circ \big(K(\partial_{n-1}^+)-K(\partial_{n-1}^-)\big)$ 
    is trivial (zero) on these elements. Since
    \[\big(K(\partial_n^+)-K(\partial_n^-)\big)\big[[x_0,x_1,\dots,\widehat{x_i},\ldots,x_n],0\big] =\sum_{i=0}^n (-1)^i \big[[x_0,x_1,\dots,x_n],0\big],\]
    the proof is now a straightforward computation analogous to the standard one for chain complexes in simplicial homology.
\end{proof}

Proposition \ref{proposition:semiringchaincomplex} allows us to define the homology of a directed simplicial complex with coefficients on a semiring.

\begin{definition}\label{definition:semiringHomologyOfDirectedComplex}
	Let $X$ be a directed simplicial complex and $n \ge 0$ and $\Lambda$ a cancellative semiring. The $n$-dimensional \emph{homology} of $X$, written $H_n(X,\Lambda)$, is the $n$th homology $\Lambda$-semimodule of the chain complex $\{C_n(X,\Lambda),\partial_n^+,\partial_n^-\}$.
\end{definition}

\begin{remark}%\label{remark:orderedTupleAssociated}
	If $\Lambda$ is a ring, by defining $\partial = \partial^+-\partial^+$, $\{C_n(X,\Lambda),\partial\}$ is a chain complex in the usual sense. In fact, it is precisely the chain complex introduced in Definition \ref{definition:chainComplexDirectedComplex}. Thus, as a consequence of Remark \ref{remark:homologysemiringsandrings}, the homology theory introduced in Definition \ref{definition:semiringHomologyOfDirectedComplex} coincides with the one introduced in Definition \ref{definition:directedHomologyComplexes}. Therefore, homology with semiring coefficients of directed simplicial complexes is a non-trivial generalisation of homology with ring coefficients.
\end{remark}

The following result will become useful later on, so we recorded it here. 

\begin{proposition}\label{proposition:homologyandK}
	Let $\Lambda$ be a cancellative semiring and $X$ a directed simplicial complex. The chain complexes of $K(\Lambda)$-semimodules $\big\{K\big(C_n(X,\Lambda)\big),K(\partial_n^+),K(\partial_n^-)\big\}$ and $\big\{C_n\big(X,K(\Lambda)\big),\partial_n^+,\partial_n^-\big\}$ are isomorphic.
\end{proposition}

\begin{proof}
	Take $n\ge 0$ an integer. By definition, $C_n\big(X,K(\Lambda)\big)$ is the free $K(\Lambda)$-module over the  elementary $n$-chains $[x_0,x_1,\dots,x_n]$. Similarly, $K\big(C_n(X,\Lambda)\big)$ is a free $K(\Lambda)$-module with a basis given by the elements $\big[[x_0,x_1,\dots,x_n],0\big]$ where $[x_0,x_1,\dots,x_n]$ is an elementary $n$-chain. Consider the $K(\Lambda)$-morphisms
	\begin{align*}
		\begin{aligned}
		\alpha_n\colon C_n\big(X,K(\Lambda)\big) & \longrightarrow K\big(C_n(X,\Lambda)\big) \\
		[x_0,x_1,\dots,x_n] & \longmapsto \big[[x_0,x_1,\dots,x_n],0\big],
		\end{aligned} &&
		\begin{aligned}
		\beta_n\colon K\big(C_n(X,\Lambda)\big) & \longrightarrow C_n\big(X,K(\Lambda)\big) \\
		\big[[x_0,x_1,\dots,x_n],0\big] & \longmapsto [x_0,x_1,\dots,x_n].
		\end{aligned}
	\end{align*}
	Immediate computations show that the families of $K(\Lambda)$-morphisms $\{\alpha_n\}$ and $\{\beta_n\}$ are morphisms of $K(\Lambda)$-semimodules which are inverses of each other, and the claim follows.
\end{proof}

\subsection{Functoriality of semiring homology}

In this section, we prove that homology is a functor from the category of directed simplicial complexes to the category of graded $\Lambda$-semimodules. We also show that two morphisms allowing for the construction of the prism operator must induce the same map on homology, a result we will need to prove our persistent homology stability result. Recall the definition of morphism of directed simplicial complexes, Definition \ref{definition:simplicialMorphism}.

\begin{definition}
	Let $f\colon X\to Y$ be a morphism of directed simplicial complexes. Then $f$ induces morphisms of $\Lambda$-semimodules $C(f) = \{C_n(f)\}$,
	\begin{align*} C_n(f)\colon C_n(X,\Lambda)&\longrightarrow C_n(Y,\Lambda) \\
		[x_0,x_1,\dots,x_n]&\longmapsto [f(x_0), f(x_1),\dots,f(x_n)].
	\end{align*}
	We will often abbreviate $C_n(f)=f_n$.
\end{definition}

\begin{proposition}\label{proposition:semimodulesdegmorphism}
	If $f\colon X\to Y$ is a morphism of directed simplicial complexes, the family of maps $\{C_n(f)\}$ is a $\Lambda$-homomorphism of chain complexes. Therefore, it induces a map $H_n(f)\colon H_n(X,\Lambda)\to H_n(Y,\Lambda)$.
\end{proposition}

\begin{proof}
    Simple computations show that $f_{n-1}\partial_n^+ = \partial_n^+ f_n$ and that $f_{n-1}\partial_n^- = \partial_n^- f_n$.
\end{proof}

\begin{remark}\label{remark:KCnf}
	As a consequence of Proposition \ref{proposition:homologyandK}, the map $C_n(f)\colon C_n\big(X,K(\Lambda)\big)\to C_n\big(Y,K(\Lambda)\big)$ is precisely $K\big(C_n(f)\big)\colon K\big(C_n(X,\Lambda)\big)\to K\big(C_n(Y,\Lambda)\big)$.
\end{remark}

\begin{corollary}
	Homology is a functor from the category of directed simplicial complexes to the category of graded $\Lambda$-semimodules. In particular, isomorphic directed simplicial complexes have isomorphic homologies. Note further that when $\Lambda$ is a ring, the homology functor is precisely the functor introduced in Section \ref{section:homologyFunctoriality}.	
\end{corollary}

We finish this section by showing a sufficient condition for two morphisms to induce the same map on homology. We will need this result to prove that our definition of persistent homology is stable.

\begin{lemma}%\label{lemma:prismConstruction}
	Let $\Lambda$ be a cancellative semiring. Let $X$ and $Y$ be two directed simplicial complexes and let $f,g\colon X\to Y$ be morphisms of directed simplicial complexes such that if $(x_0,x_1,\dots,x_n)\in X$, then $\big(f(x_0),f(x_1),\dots,f(x_i),g(x_i),\dots,g(x_n)\big)\in Y$ for every $i=0,1,\dots,n$. Then, $H_n(f) = H_n(g)$ for every $n \ge 0$.
\end{lemma}

\begin{proof}
	For $x = [x_0,x_1,\dots,x_n]\in C_n(X,\Lambda)$ an elementary $n$-chain, define
	\begin{align*}
	s_n^+[x_0,x_1,\dots,x_n] &= \sum_{i=0}^{\lfloor\frac{n}{2}\rfloor} [f(x_0),\dots,f(x_{2i}),g(x_{2i}),\dots,g(x_n)], \text{ and}\\
	s_n^-[x_0,x_1,\dots,x_n] &= \sum_{i=0}^{\lfloor\frac{n-1}{2}\rfloor} [f(x_0),\dots,f(x_{2i+1}),g(x_{2i+1}).\dots,g(x_n)].\end{align*}
	Now recall that since $\Lambda$ is cancellative, the canonical map $k_n\colon C_n(X,\Lambda)\to K\big(C_n(X,\Lambda)\big)$ is injective for all $n$. We will show that $(s^+,s^-)\colon C(f)\simeq C(g)$ by proving that both sides of the equality in Definition \ref{definition:chainHomotopy} have the same images through $k_n$. We will also make use of the isomorphism $K\big(C_n(X,\Lambda)\big) \cong C_n\big(X,K(\Lambda)\big)$ established in Proposition \ref{proposition:homologyandK}, and that for any $\Lambda$-morphism $h\colon C_n(X,\Lambda)\to C_n(X,\Lambda)$, $k_n\circ h = K(h)\circ k_n$.

	By abuse of notation, we write $k_n(x) = x = [x_0,x_1,\dots,x_n] \in C_n\big(X,K(\Lambda)\big)$. Denote $\partial_n = K(\partial_n^+) - K(\partial_n^-)$ and $s_n = K(s_n^+) - K(s_n^-)$. Then,
	\[\partial_n(x) = \sum_{i=0}^n (-1)^i [x_0,x_1,\dots,\widehat{x_i},\dots,x_n], \hspace{0.2cm} s_n(x) = \sum_{j=0}^n (-1)^j \big[f(x_0),f(x_1),\dots,f(x_j),g(x_j),\dots,g(x_n)\big].\]
	Note that, by appropriately grouping the terms in the equality in Definition \ref{definition:chainHomotopy}, it suffices to show that $K(g_n)-K(f_n)$ and $\partial_{n+1} s_n + s_{n-1}\partial_n$ have the same image on $x$. These computations are analogous to the ones in Lemma \ref{lemma:prismConstruction}.
\end{proof}

\subsection{Comparison to directed homology and relation to persistence}

In Section \ref{section:homologySimplicialComplexes}, directed homology modules were introduced by taking submodules of homology generated by cycles with only non-negative coefficients on the elementary chains. This was done for the rings $\mathbb{Z}$, $\mathbb{Q}$, and $\mathbb{R}$. These three rings have in common that they are the Grothendieck's completion of a cancellative, zerosumfree semiring. Namely, $\mathbb{Z} = K(\mathbb{N})$, $\mathbb{Q}=K(\mathbb{Q}^+)$, $\mathbb{R} = K(\mathbb{R}^+)$. It turns out that homology with coefficients in this class of semirings has the ability to detect directed cycles, and what we have introduced in the main body of this article is a particular case of this fact.

\begin{proposition}\label{proposition:directedCyclesCorrespondence}
    Let $\Lambda\in \{\mathbb{N},\mathbb{Q}^+,\mathbb{R}^+\}$ and let $X$ be a directed simplicial complex. Consider $Z_n^{\Dir}\big(X,K(\Lambda)\big)$ the set of directed cycles introduced in Definition \ref{definition:directed-homology}. Then, $Z_n^{\Dir}\big(X,K(\Lambda)\big) \cong Z_n(X,\Lambda)$ as $\Lambda$-semimodules.
\end{proposition}

\begin{proof}
    First note that for the considered semirings, $\Lambda$ can be seen as the non-negative elements of $K(\Lambda)$. Using this identification, it is straightforward to show that $Z_n^{\Dir}\big(X,K(\Lambda)\big)$ is a $\Lambda$-semimodule. In particular, linear combinations of directed cycles with non-negative coefficients are directed cycles.
    
    By abuse of notation, we will denote elementary $n$-chains in the same way in $C_n(X,\Lambda)$ and in $C_n\big(X,K(\Lambda)\big)$. Assume that $\sum_{i=1}^n \lambda_i c_i\in Z_n(X,\Lambda)$. Then clearly $\sum_{i=1}^n \lambda_i c_i\in Z_n\big(X,K(\Lambda)\big)$ (see Proposition \ref{proposition:homologyandK}). Furthermore, since the involved coefficients are non-negative, $\sum_{i=1}^n \lambda_i c_i\in Z_n^{\Dir}\big(X,K(\Lambda)\big)$. Consequently, we can introduce a map
    \begin{align*}
        \varphi\colon Z_n(X,\Lambda) &\longrightarrow Z_n^{\Dir}\big(X,K(\Lambda)\big)\\
        \sum_{i=1}^n \lambda_i c_i &\longmapsto \sum_{i=1}^n \lambda_i c_i.
    \end{align*}
    It is immediate to check that $\varphi$ is a morphism of $\Lambda$-semimodules. It is also injective. We only need to prove that it is surjective. For that, it would be enough to prove that if $\sum_{i=1}^n \lambda_i c_i\in Z_n^{\Dir}\big(X,K(\Lambda)\big)$, then $\sum_{i=1}^n \lambda_i c_i\in Z_n(X,\Lambda)$. To that purpose, note that $x =\sum_{i=1}^n \lambda_i c_i$ can be regarded as an element of $C_n(X,\Lambda)$ as the involved coefficients are non-negative, and thus, in $\Lambda$. We only need to prove that $\partial^+(x) = \partial^-(x)$. Since $\Lambda$ is cancellative, it is enough to prove that $K\big(\partial^+(x)\big) = K\big(\partial^-(x)\big)$, or equivalently, that $K\big(\partial^+(x)\big) - K\big(\partial^-(x)\big) = 0$. But this follows from Proposition \ref{proposition:homologyandK} and from the fact that $x\in Z_n\big(X, K(\Lambda)\big)$.
\end{proof}

In general, we could use homology with coefficients in a cancellative, zerosumfree semiring to detect directed cycles. Namely, we can prove an analogous to Proposition \ref{proposition:homologyPolygons} for homology with semiring coefficients.

\begin{proposition}\label{proposition:homologyPolygonsAppendix}
    Let $\Lambda$ be a cancellative, zerosumfree semiring. Let $X$ be a $1$-dimensional directed simplicial complex with vertex set $\{v_0,v_1,\dots,v_n\}$, $n\ge 1$, and with $1$-simplices $e_0$, $e_1,\dots, e_n$ where either $e_i = (v_i, v_{i+1})$ or $e_i = (v_{i+1}, v_i)$, $i=0,1,\dots,n$, and where $v_{n+1}=v_0$. Then, $H_1(X,\Lambda) = \Lambda$ is non-trivial if and only if either $e_i = (v_i,v_{i+1})$ for all $i$ or $e_i = (v_{i+1},v_i)$ for all $i$, and $H_1(X,\Lambda)=\{0\}$ in any other case.
\end{proposition}

\begin{proof}
    Let $[e_i]$ denote the elementary $1$-chain associated to $e_i$. Let $x = \sum_{i=0}^n \lambda_i [e_i]$ be any non-trivial $1$-cycle. We can assume without loss of generality that $\lambda_0\ne 0$. If $e_0 = (v_0,v_1)$, then $\lambda_0[v_1]$ is a non-trivial summand in $\partial^+(e_0)$. Since $\Lambda$ is zerosumfree, such summand does not have an inverse, thus $\lambda_0[v_1]$ must be a summand in $\partial^-(x)$. Now note that $e_1$ is the only other simplex involving the vertex $v_1$. Furthermore, $[v_1]$ appears in $\partial^-[e_1]$ if and only if $e_1 = (v_1,v_2)$, in which case $\partial^-[e_1]=\lambda_1 v_1$. We further deduce that $\lambda_1=\lambda_0$.
    
    By proceeding iteratively, we deduce that if $x$ is non-trivial, necessarily $e_i = (v_i, v_{i+1})$ and $\lambda_0 = \lambda_1 = \dots = \lambda_n = \lambda$, in which case $x = \sum_{i=0}^n \lambda [v_i, v_{i+1}]$. Simple computations show that such $x$ is indeed a cycle, and since there are no $2$-simplices, it cannot be homologous to any other cycle. Thus $H_1(X,\Lambda)=\Lambda$. The case where $e_0 = (v_1,v_0)$ is analogous.
\end{proof}

The question would then be how to use this information in the setting of persistent homology. Recall from Remark \ref{remark:homologyandKfunctor} that the family of canonical maps $k_{C(X,\Lambda)}\colon C_n(X,\Lambda)\to K\big(C_n(X,\Lambda)\big)$ gives rise to a morphism of $\Lambda$-semimodules $k_{C(X,\Lambda)}\colon H_n(X,\Lambda)\to H_n\big(K(C(X,\Lambda))\big)$,
which takes the class of $x$ to the class of $[x,0]$. As a consequence of Proposition \ref{proposition:homologyandK}, we can also rewrite $H_n\big(K(C(X,\Lambda))\big) = H_n\big(C(X,K(\Lambda))\big)$. Therefore, the map $k_{C(X,\Lambda)}$ above can be regarded as a morphism $k_{C(X,\Lambda)}\colon H_n(X,\Lambda)\to H_n\big(X,K(\Lambda)\big)$. We denote this map by $k_X$. We then have the following result.

\begin{proposition}\label{proposition:directedHomologyFormulations}
    Let $X$ be a directed simplicial complex and let $\Lambda\in \{\mathbb{N},\mathbb{Q}^+,\mathbb{R}^+\}$. The directed homology module introduced in Definition \ref{definition:directed-homology}, $H_n^{\Dir}\big(X,K(\Lambda)\big)$, is the submodule of $H_n\big(X,K(\Lambda)\big)$ generated by $\Im k_X$.
\end{proposition}

\begin{proof}
    Note that all representatives of classes in $H_n(X,\Lambda)$ are in $Z_n(X,\Lambda)$. Thus, the image through $k_X$ corresponds to a subset of the classes generated by positive cycles, as a consequence of Proposition \ref{proposition:directedCyclesCorrespondence}. Namely, $\Im k_X\subseteq H_n^{\Dir}\big(X, K(\Lambda)\big)$. To prove the result, it would be enough to prove that every class of a directed cycle is in $\Im k_X$. But this follows from the definition of $H_n(X,\Lambda)$.
\end{proof}

More generally, given $\Lambda$ a cancellative, zerosumfree semiring, we can define $H_n^{\Dir}\big(X,K(\Lambda)\big)$ as the submodule of $H_n\big(X,K(\Lambda)\big)$ generated by $\Im k_X$, and such module would detect directed cycles as exhibited by Proposition \ref{proposition:directedHomologyFormulations}. If $K(\Lambda)$ is a field, we can also use these homology modules to define persistence vector spaces for which we can prove stability results in a way analogous to the results in Section \ref{section:directedPersistentHomology}.

The results of this paper arose from the ideas in the constructions above, which are introduced in their full generality using homology with semiring coefficients. This shows that the idea of introducing a submodule of directed homology by taking cycles with non-negative coefficients has a deeper algebraic meaning, providing more context for the results in this paper and presenting them in their full generality. Furthermore, similar ideas could be considered for semirings which are not zerosumfree, perhaps allowing for the detection of additional information that persistent homology with ring coefficients may miss. The matter of whether a theory of persistence semimodules could provide more information could also be looked into. %Finally, regarding the computational challenges introduced in Section \ref{section:algorithms}, it is possible that tools from semiring theory could be used to produce a more efficient algorithm for the computations involved in the proposed theory of directed persistent homology.

%\nocite{*}

\setlength{\parindent}{0cm}

\end{document}